\documentclass{amsart}
\usepackage{mathrsfs,dsfont,textcomp}
\usepackage[all]{xy}

\usepackage{color}

\newtheorem{theorem}{Theorem}[section]
\newtheorem{lemma}[theorem]{Lemma}
\newtheorem{proposition}[theorem]{Proposition}
\newtheorem{corollary}[theorem]{Corollary}
\newtheorem{propdef}[theorem]{Proposition-Definition}

\theoremstyle{definition}
\newtheorem{definition}[theorem]{Definition}
\newtheorem{example}[theorem]{Example}

\theoremstyle{remark} 
\newtheorem{remark}[theorem]{Remark}

\numberwithin{equation}{section}

\begin{document}

\title[Moving planes, Jacobi curves and Finsler geometry] {Moving planes, Jacobi curves and the dynamical approach to Finsler geometry}

\author[C. Dur\'an]{Carlos Dur\'an}
\address{Departamento de Matem\'atica, \hfill\break\indent
Universidade Federal de Paran\'a, \hfill\break\indent
Setor de Ci\^encias Exatas, Centro Polit\'ecnico, 
 Caixa Postal 019081,  CEP 81531-990, \hfill\break\indent
Curitiba, Paran\'a, Brazil}
\email{cduran@ufpr.br}

\author[H. Vit\'orio]{Henrique Vit\'orio}
\address{Departamento de Matem\'atica, \hfill\break\indent
Universidade Federal de Pernambuco, \hfill\break\indent
Cidade Universit\'aria, \hfill\break\indent
Recife, Pernambuco, Brazil}
\email{henriquevitorio@dmat.ufpe.br}




\thanks{This work was supported by {\rm CNPq}, grant {\rm No.} 232664/2014-5}




\begin{abstract}
We express invariants of Finsler manifolds in a geometrical way by means  of using {\em moving planes} and their associated {\em Jacobi curves},
which are curves in a fixed homogeneous Grassmann manifold. Some applications are given.

\end{abstract}

\maketitle

\section{Introduction}

A common way of writing computations in Finsler geometry is through some extension of the Levi-Civita calculus of Riemannian geometry. However, since there cannot be a 
Levi-Civita connection in Finsler geometry (for reasonable notions of connection with metric comptibility and torsion freeness, if a Finsler manifold admits such a connection it is actually Riemannian), there is a plethora of connections 
(Berwald, Cartan, Chern and Rund, \dots) where each one of them is defined by {\em partial} compatibilities and torsion freeness. 
While this connection formalism has led to important developments, there are contexts where a different point of view can shed new light.
\par An alternative approach, of a more dynamical flavor, to the geometry of sprays and Finsler metrics consists in regarding the local differential invariants of
sprays and Finsler metrics as local invariants (under the action of the appropriate group of diffeomorphisms) of the following type of geometric data on a manifold:

\begin{definition}\label{definitionmovingplane}
A {\it moving plane} on a smooth manifold $X$ is a triplet $\mathscr{P}=(\Delta_r,\Delta_k,\Phi_t)$, where 
\begin{enumerate}
\item $\Delta_k\subset\Delta_r$ are distributions on $X$ with dimensions $k$ and $r$, respectively. 
\item $\Phi_t$ is a flow in $X$ which leaves $\Delta_r$ invariant.
\end{enumerate}
\end{definition}
For instance, the prototypical examples that motivated this paper are the cases where (we refer to $ \S$ \ref{ssectionSpraysFinsler} for precise definitions)
\begin{enumerate}
\item  $X$ is the tangent bundle without the zero section $TM\backslash 0$ of a manifold $M^n$, $\Delta_{2n}$ is the full tangent distribution, $\Delta_n$ is the
vertical distribution $\mathcal{V}TM$, and $\Phi_t$ is the flow corresponding to a spray $S$ on $M$.
\item $X$ is the unit co-sphere bundle $\Sigma^*_FM$ of a Finsler manifold $(M^n,F)$, $\Delta_{2n-2}$ is the canonical contact distribution on $\Sigma^*_FM$, 
$\Delta_{n-1}$ is the vertical distribution $\mathcal{V}\Sigma^*_FM$ and $\Phi_t$ is the restriction to $\Sigma^*_FM$ of the co-geodesic flow of $F$.
\end{enumerate}
This approach is implicit in the pioneering works of Grifone \cite{Grifone} and Foulon \cite{Foulon} where, for instance, the classical notions of Ehresman connection and
curvature endomorphism from the theory of second order differential equations and Finsler metrics, are recovered  by considering the so-called almost tangent structure
(in the case \cite{Grifone}) and the vertical endomorphism (in the case \cite{Foulon}) and their successive Lie derivatives along the geodesic vector field.
\par 
Back to the moving plane setting, the infinitesimal action of the flow $\Phi_t$ on the distribution $\Delta_k$ gives rise, for each $x\in X$, to a curve
\begin{equation}\label{jacobicurvedefinition}
\ell_x(t)  =  ({\Phi_t }^*\Delta_k)(x)  =    {\rm d}\Phi_{-t}(\Phi_t(x))\Delta_k(\Phi_t(x))
\end{equation}
of $k$-dimensional subspaces of the fixed vector space $\Delta_r(x)$; that is, $\ell_x(t)$ is a curve on the Grassmannian manifold ${\rm Gr}_k(\Delta_r(x))$, called the
{\it Jacobi curve} of $\mathscr{P}$ based at $x$. In the above
examples, the Jacobi curves live on half-Grassmannians ${\rm Gr}_n(\mathds{R}^{2n})$ and on Lagrangian Grassmannians $\Lambda(\mathds{R}^{2n})$, respectively.
It is well-known that the topology of curves of Lagrangian subspaces successfully describes conjugacy of geodesics via the Maslov index theory \cite{PT}. As we will show here,
the {\em local} geometry of Jacobi curves also describes  relevant local invariants of sprays and Finsler metrics, in particular the invariants related to variational phenomena; by this we mean, for example, the Jacobi endomorphism 
$Y\mapsto 
R(Y,T)T$
which appears in the Jacobi equation and leads to the definition of flag curvature.
\par
To the best of our knowledge, this was first noticed by Adhout \cite{Ah} in the case of 
Riemannian geodesic flows; there, by identifying an important generic property of curves of Lagrangian subspaces (the {\it fanning} property, later extended to curves
on ${\rm Gr}_n(\mathds{R}^{2n})$ in \cite{AD1} ), the author uncovers the local invariants as linear symplectic invariants of the Jacobi curve. On the other hand, the geometry of curves on
${\rm Gr}_n(\mathds{R}^{2n})$ and $\Lambda(\mathds{R}^{2n})$, under the action of the general linear and symplectic groups, is a beautiful subject in itself. As has been shown
in \cite{AD1}, the behaviour of the class of fanning curves can be completely described, in the spirit of Cartan-Klein, by a set of linear invariants. As we shall show here, the formalism of
\cite{AD1} applied to the Jacobi curves of the above examples gives us the desired local invariants. This gives a unified treatment of the approaches of Grifone, Foulon and Adhout, and 
can be viewed as a Cartan-Klein geometrization of them. This point of view leads to some applications to Finsler geometry that we now describe:

 \subsection*{An O'Neill formula for Finsler submersions.} A fundamental tool in the study of curvature properties of Riemannian manifolds is the {\em O'Neill tensors} 
and associated {\em O'Neill formulas} \cite{oneill}, which relate curvatures of the total space 
and the base of Riemannian submersions; see for example \cite{gromoll-walschap} for a description of its use in the study of non-negative curvature. We give an O'Neill formula for Finsler manifolds expressed in terms of invariants of the Jacobi curve. As is common in Finsler geometry, the results are interesting even for Riemannian manifolds: the standard proof and applications of O'Neill formulas are given as algebraic manipulations of the Levi-Civita connection, whereas the Jacobi curve gives an O'Neill formula as a quantification of the relationship, as a symplectic reduction, of the geodesic flows of the total space and the base \cite{AD2}. In addition to curvature bounds applications, the fine details 
of the O'Neill tensor allows the consideration of rigidity results of special submersions \cite{duran-speranca,gg} and the original rigidity results of O'Neill
(theorem 4 of \cite{oneill}),  which would be quite interesting to generalize to the Finslerian setting.

\subsection*{A characterization of the sign of flag curvature.} An important area or Riemannian geometry is the construction of examples of manifolds with sign properties 
of the sectional curvature, for example manifolds of positive (resp. negative) sectional  curvature and their associated relaxed conditions non-negative (resp. non-positive), see e.g. \cite{ziller-examples}. This interest has spread to Finsler manifolds \cite{rademacher}, and the study of examples  has begun with the homogeneous case (see \cite{XD} for a survey). We give a dynamical characterization of the sign of flag curvature in terms of the Jacobi curve, or, more precisely, in terms of the {\em horizontal curve}, which is another curve in the (Lagrangian) Grassmannian  canonically produced from the Jacobi curve.

\subsection*{The flag curvature of a class  of projectively related Finsler metrics.} One area where Finsler geometry is completely different from Riemannian geometry is {\em inverse problems}, where in the Finsler case there is typically a 
rich moduli space (specially in the non-symmetric case), whereas there is rigidity in the Riemannian case, for example, 
in Hilbert's Fourth Problem \cite{JC} and projectively flat metrics of constant curvature \cite{bryant-flat-constant}. In this spirit, two Finsler metrics are {\em projectively related} if they share the same geodesics up to reparametrization. An important transformation that does not change the projective class of a metric is the addition of a closed 1-form. We describe how the Jacobi curve furnishes a formula relating the flag curvature of a metric with that of its deformation by a closed 1-form.

\subsection*{The flag curvature of Katok perturbations.}
In 1973 A. Katok constructed examples of a non-symmetric Finsler metric on the sphere $S^2$ with 
only two prime closed geodesics; the geometry of these metrics has been nicely described in \cite{Ziller} and a standard Finsler description is given in
\cite{shen2} . It is well-known that these metrics have constant 
curvature (see, e.g. Foulon \cite{FR} or $\S$11 of Rademacher \cite{rademacher}). We  present a proof of this property, due to J.C. \'Alvarez, that proceeds  
by showing that the Jacobi curves of the original metric and of the Katok-perturbed one are equivalent under a linear-symplectic transformation, thus having the same invariants.
\begin{remark}
The local geometry of the Jacobi curve has also been intensively studied with motivation coming from Control Theory and Sub-Riemannian geometry; see \cite{ABR} and the references therein for a contemporary account, and the appendix of \cite{AD1} for comparison of the approaches to the invariants. In particular, in \cite{ACZ} , there is a reduction procedure similar to the one we use for giving the Finslerian version of the O'Neill tensor and associated formula.
\end{remark}

\begin{remark} Moving planes and their Jacobi curves in half-Grassmannians are specially adapted to Finsler geometry; however this concept can be
generalized and applied to other situations: one can consider for example  a whole linear flag of distributions $\Delta_{k_1}\subset\Delta_{k_2} \subset \dots \subset\Delta_r$, and its associated Jacobi curve in a fixed  flag manifold. This situation appears in
the study of higher order variational problems, where the $\Delta_{k_i}$ are kernels 
of the derivative of the projections of the jet spaces of curves 
$\pi: J^s(\mathbb{R}, M) \to J^r(\mathbb{R}, M)$ for adequate $s>r$. See \cite{crampin-sarlet-cantrijn, leon-rodriguez, duran-otero, duran-peixoto}. 
\end{remark}

After this introduction, the paper is organized as follows: we give some preliminaries in $\S$\ref{sectionprelims}  in order to fix language and make the paper reasonably self-contained. In $\S$\ref{sectionmovingplanes}
 we establish how curvature invariants are expressed in terms of moving planes and their associated Jacobi fields, by relating these invariants with those obtained by the dynamic method and Finsler connections; in particular, we recover the flag curvature in Theorem 
 \ref{theoremjacobiflag}. The rest of the sections of  the paper correspond to each of the aforementioned applications. 
\medskip 

\noindent{\bf Acknowledgments.}
We thank Juan Carlos Alvarez Paiva for his participation in the early stages of the research presented in section \ref{sectionmovingplanes} and for allowing us to present his unpublished proof of curvature invariance of Katok deformations; without this and his gentle prodding this paper would not exist. The second author would like to thank the support provided by the 
Mathematischen Instituts der Universit\"at Leipzig, where part of this work was done, and the financial support provided by the Brazilian program Science Without Borders, Grant No. 232664/2014-5.

\section{Preliminaries} \label{sectionprelims}

The two ends that this paper aims to connect are, on one side, the global invariants of Finsler manifolds, and on the other side, the invariants of curves in a fixed Grassmann manifold viewed as a homogenous space. Sections \ref{ssectionSpraysFinsler} 
and \ref{fanningcurvessection} correspond respectively to the necessary preliminaries of each side.

\subsection{Sprays and Finsler manifolds} \label{ssectionSpraysFinsler}

\subsubsection{Notations and the structure of the tangent bundle} \label{sssectionStructureTangent}

We shall denote by $TM\backslash 0$ the tangent bundle of a manifold with the null section removed, and by $\pi$ and $\rho$
the projection maps $\pi:TM\backslash 0\rightarrow M$, $\rho:T(TM\backslash 0)\rightarrow TM\backslash 0$. The latter contains as
a vector subbundle the {\it vertical tangent bundle} 
\begin{equation}\label{verticalbundle}
\rho:\mathcal{V}TM\rightarrow TM\backslash 0, 
\end{equation}
whose fibers are the tangent spaces of
the fibers of $\pi$. We shall call {\it vertical vector fields} on $TM\backslash 0$ the sections of (\ref{verticalbundle}). 
The {\it vertical lift} at a given $w\in TM\backslash 0$ is the tautological isomorphism 
\begin{equation}\label{tautologicalisomorphism}
i_w:T_{\pi(w)}M\rightarrow\mathcal{V}_w TM~,~~ i_w(u)=({\rm d}/{\rm d}t)|_{t=0}(w+t\cdot u),
\end{equation}
where the name of these isomorphisms stems from the fact that 
$i_w$ furnishes canonical lifts of a vector fields $U$ on $M$ to vertical fields 
$U^\mathfrak{v}$on $TM\backslash 0$; the same procedure  also gives vertical lifts of vector fields defined along curves in $M$. The {\it canonical vector field}
$C$ on $TM\backslash 0$ is defined by $C_w=i_w(w)$.
\par We remark that analogous constructions apply to the punctured co-tangent bundle $\tau:T^*M\backslash 0\rightarrow M$:
a vertical distribution $\mathcal{V}T^*M$ on $T^*M\backslash 0$, tautological isomorphisms $i_\xi:T^*_{\tau(\xi)}M\rightarrow\mathcal{V}_\xi T^*M$, and the canonical vector field $C^*$
on $T^*M\backslash 0$ are defined
as before.
\par With this tool in hand, we can define 

\begin{definition}
The {\it almost-tangent structure} of $TM\backslash 0$ is the section $\mathcal{J}$ 
of \newline ${\rm End}(T(TM\backslash 0))\rightarrow TM\backslash 0$ defined
by
\begin{equation}\nonumber
\mathcal{J}(X)=i_w({\rm d}\pi(w)X),~~~\mbox{for $w=\rho(X)$}.
\end{equation}
\end{definition}

Observe that $\mathcal{J}$ has both kernel and image equal to $\mathcal{V}TM$.

\begin{definition}
A {\it second order differential equation} (SODE) on $M$ is a smooth vector field $S$ on $TM\backslash 0$ such that $\mathcal{J}(S)=C$. This means that
the integral curves of $S$ are of the form $t\mapsto\dot{\gamma}(t)$, for some class of curves $\{\gamma\}$ in $M$.
If furthermore $[C,S]=S$, then $S$ is called a {\it spray}, in which case the curves $\{\gamma\}$ are the geodesics of $S$.
\end{definition}

In {\it natural local coordinates} $(x,y)$ for $TM\backslash 0$, i.e. $(x,y)$ are induced from local coordinates $x$ for $M$, a SODE assumes the form
\begin{equation}\label{SODE}
S = \sum_i y^i\frac{\partial}{\partial x_i}-2\sum_iG_i(x,y)\frac{\partial}{\partial y^i}, 
\end{equation}
for certain smooth functions $G_i$ that are positively homogeneous of degree 2 in $y$ if, and only if, $S$ is a spray. The following basic 
property (cf. \cite[Prop. I.7]{Grifone}) will be essential later. For the sake of completeness we have included a proof.

\begin{lemma}\label{lemmaalmosttangent}
If $S$ is a {\rm SODE} on $M$ and $X$ is a vertical vector field on $TM\backslash 0$, then $\mathcal{J}([X,S])=X$. 
\end{lemma}

\begin{proof}
Since $\mathcal{J}$ vanishes on vertical vectors, $\mathcal{J}([X,S])-X$ is $C^\infty(TM\backslash 0)$-linear in the sections $X$ of (\ref{verticalbundle}).
Relatively to natural local coordinates $(x,y)$, $\mathcal{J}(\partial/\partial y_i)=0$, $\mathcal{J}(\partial/\partial x_i)=\partial/\partial y_i$, and, from
(\ref{SODE}), $[\partial/\partial y_i,S]=\partial/\partial x_i-2\sum_j(\partial G_j/\partial y_i)\partial/\partial y_j$. Therefore, 
$\mathcal{J}([\partial/\partial y_i,S])-\partial/\partial y_i=0$.
\end{proof}

\subsubsection{Finsler manifolds}

\begin{definition}
A Finsler metric on a smooth manifold $M$ is a function $F:TM\rightarrow[0,\infty)$ that is smooth on $TM\backslash 0$ and that restricts to a
Minkowski norm $F_m$ on each tangent space $T_mM$. This means that
\begin{itemize}
\item[$(i)$] $F(v)=0$ if, and only if $v=0$;
\item[$(ii)$] $F(\lambda v)=\lambda F(v)$, if $\lambda\geq 0$;
\item[$(iii)$] For every $v\in TM\backslash 0$, the second fiber-derivative of $(1/2)F^2$ at $v$,
\begin{equation}\label{fundamentaltensor}
g_F(v) := (1/2){\rm d}^2_fF^2(v),
\end{equation}
is a positive-definite inner product on $T_{\pi(v)}M$.
\end{itemize}
The inner product (\ref{fundamentaltensor}) is usually referred to as the {\it fundamental tensor} of $F$ at a $v$. 
\end{definition}

An important concept attached to a Finsler metric is the notion of {\it dual}.

\begin{definition}
The dual of a Finsler metric $F$ on $M$ is the function $F^*:T^*M\rightarrow\mathds{R}$ obtained by fiberwise taking the dual of the Minkowski norms
$F_m$, that is,
\begin{equation}\nonumber
(F^*)_m(\xi) = \sup_{F_m(v)=1}\xi(v). 
\end{equation}
\end{definition}

Alternatively, the dual $F^*$ of $F$ is obtained by composing $F$ with the inverse of its {\it Legendre transformation}. The latter is the diffeomorphism
$\mathscr{L}_F:TM\backslash 0\rightarrow T^*M\backslash 0$, 
\begin{equation}\label{legendretransf}
\mathscr{L}_F(v) = (1/2){\rm d}_fF^2(v) = g_F(v)(v,\hskip 1pt\cdot\hskip 1pt). 
\end{equation}
We remark that, from the homogeneity, the fiber derivative of $\mathscr{L}_F$ is given by
\begin{equation}\label{fiberderivative}
{\rm d}_f\mathscr{L}_F(v)w = g_F(v)(w,\hskip 1pt\cdot\hskip 1pt). 
\end{equation}

\subsubsection{The Hamiltonian point of view}.

\vskip 6pt

\noindent{\it The co-tangent bundle setting.}
Let us begin by recalling
\begin{definition}
The {\it canonical 1-form} of $T^*M$ is the 1-form $\alpha$ on $T^*M$ defined by
\begin{equation}\label{alpha}
\alpha_\xi(X)=\xi({\rm d}\tau(\xi)X).
\end{equation}
The 2-form $\omega=-{\rm d}\alpha$ defines the so-called
{\it canonical symplectic structure}
of $T^*M$.
\end{definition}
\noindent We remark that one can recover $\alpha$ from $\omega$ and the tautological vector field $C^*$ via
\begin{equation}\label{alphaomega}
\alpha = -i_{C^*}\omega. 
\end{equation}

Given a Finsler metric $F$, let us consider the Hamiltonian function
\begin{equation}\label{hamiltoniandual}
(1/2)(F^*)^2 : T^*M\backslash 0\rightarrow\mathds{R} 
\end{equation}
on the symplectic manifold $(T^*M\backslash 0,\omega)$.

\begin{definition}
We shall call {\it co-geodesic vector field} of $F$, and denote by $S^*_F$ (or simply $S^*$), the Hamiltonian vector field of (\ref{hamiltoniandual}); that is,
$S^*_F$ is the vector field on $T^*M\backslash 0$ defined
by
\begin{equation}\nonumber
(1/2){\rm d}(F^*)^2 = \omega(S^*_F,\hskip 1pt\cdot\hskip 1pt). 
\end{equation}
The corresponding flow $\Phi_t^{S_F^*}$ is the {\it co-geodesic flow} of $F$. Observe that since (\ref{hamiltoniandual}) is positively homogeneous of degree 2, then 
$[C^*,S_F^*]=S_F^*$.
\end{definition}

Being the Hamiltonian flow of (\ref{hamiltoniandual}), the co-geodesic flow preserves $\omega$ and leaves invariant every level set of $F^*$. In particular, it restricts to
a flow on the {\it unit co-sphere bundle} 
\begin{equation}\label{cospherebundle}
\Sigma^*_FM=(F^*)^{-1}(1).
\end{equation}
We shall still let $S_F^*$ and $\Phi_t^{S_F^*}$ denote their restrictions to (\ref{cospherebundle}).
In the following, $\alpha$ and $\omega$ will mean their pull-backs to (\ref{cospherebundle}).
The contact geometry of $F$ is described by

\begin{proposition}
The 1-form $\alpha$ is a contact form on $(\ref{cospherebundle})$; this means that $\omega$ is non-degenerate (hence, induces a symplectic structure) on the so-called contact distribution
${\rm ker}(\alpha)$. Furthermore, the vector field $S_F^*$ is the Reeb vector field of $(\Sigma_F^*M,\alpha)$, that is,
it is the unique vector field such that
\begin{equation}\nonumber
i_{S_F^*}\omega = 0,~~~\alpha(S_F^*) = 1. 
\end{equation}
\end{proposition}

\begin{remark}\label{remarkcontacttype}
For future reference, we remark that $T^*M$, $C^*$, $\alpha$, $\omega$, $S^*_F$, $\Sigma_F^*M$ fit in the following abstract setting.
Let $(X,\omega, S)$ be a symplectic manifold endowed with a vector field $S$ generating a symplectic flow $\Phi_t^S$, and let $\Sigma\subset X$ be a
$\Phi_t^S$-invariant hypersurface such that
\begin{enumerate}
\item $\Sigma$ is of {\it contact type} with respect to a {\it Liouville vector field} $C$; this means that (cf.\cite{dusamcduff}) $C$ is
a vector field defined in a neighborhood of $\Sigma$ that is everywhere transverse to $\Sigma$ and such that $[C,\omega]=\omega$. 
\item $S$ generates the {\it charecteristic distribution} of $\Sigma$, i.e. $T\Sigma={\rm ker}(i_S\omega|_\Sigma)$.
\item $S$ satisfies the homogeneity $[C,S]=S$.
\end{enumerate}
In this setting, 
$\alpha:=-i_C\omega$ pulls back to a contact form on $\Sigma$, still denoted by $\alpha$. Moreover, $\Phi_t^S$ restricts to
an exact contact flow on $(\Sigma,\alpha)$, i.e. $(\Phi_t^S)^*\alpha=\alpha$ for all $t$, and $-{\rm d}\alpha$ and $\omega$ restrict to the same
symplectic structure on the contact distribution ${\rm ker}(\alpha)$.
\end{remark}

\noindent{\it The tangent bundle setting.}
We shall let $\alpha_F$ and $\omega_F$ be the pull-backs of $\alpha$ and $\omega$ by the Legendre transformation $\mathscr{L}_F$. Observe that, from (\ref{legendretransf}) and
(\ref{alpha}),
\begin{equation}\nonumber
(\alpha_F)_v(X) = g_F(v)(v,{\rm d}\pi(v)X). 
\end{equation}
The pull-back of $S_F^*$ by $\mathscr{L}_F$ is a spray on $M$, the so-called {\it geodesic spray} $S_F$ of $F$, and the corresponding flow $\Phi_t^{S_F}$ is the 
{\it geodesic flow} of $F$. It follows that $\Phi_t^{S_F}$ preserves $\omega_F$ and
\begin{equation}\nonumber
\mathscr{L}_F\circ\Phi_t^{S_F} = \Phi_t^{S_F^*}\circ\mathscr{L}_F. 
\end{equation}
As in the co-tangent case, $\alpha_F$ pulls-back to a contact form, still denoted by $\alpha_F$, on the {\it unit sphere bundle} $\Sigma_FM=F^{-1}(1)$, and $S_F$ restricts
to the Reeb field of $(\Sigma_FM,\alpha_F)$, still denoted by $S_F$. The Legendre transformation $\mathscr{L}_F$ relates both contact geometries.

\subsection{The geometry of fanning curves}\label{fanningcurvessection}

In this section we summarize the invariants of curves in the half-Grassmannians and Lagrangian Grassmannians  
constructed in \cite{AD1}.

\subsubsection{Fanning curves on ${\rm Gr}_n(V)$}\label{fanningcurvessubsection}

Let $V$ be a $2n$-dimensional real vector space. A smooth curve $\ell(t)$ on the Grassmannian manifold ${\rm Gr}_n(V)$ of $n$-dimensional subspaces of $V$ is {\it fanning} if, 
upon identifying 
the tangent spaces $T_\ell{\rm Gr}_n(V)$ with the spaces of linear maps from $\ell$ to $V/\ell$, each velocity vector $\dot{\ell}(t)$ is
an invertible linear map; this is a non-degeneracy condition satisfied by an open and dense set of smooth curves. The set of fanning curves is acted upon by the general linear group ${\rm GL}(V)$ and it turns out that, with respect to the 
prolonged action of ${\rm GL}(V)$ on the space ${\rm J}^1_f(\mathds{R};{\rm Gr}_n(V))$
of one-jets of fanning curves on ${\rm Gr}_n(V)$ and the adjoint action of ${\rm GL}(V)$ on $\mathfrak{gl}(V)$, all the equivariant maps 
\begin{equation}\nonumber
{\rm J}^1_f(\mathds{R};{\rm Gr}_n(V))\rightarrow\mathfrak{gl}(V)
\end{equation} 
are of the form $a{\bf I}+b{\bf F}$, $a,b\in\mathds{R}$, where {\bf I} is the identity of $V$ and the {\it fundamental endomorphism} {\bf F} can be 
described in terms of frames as follows.

\subsubsection{Frames and the Fundamental endomorphism}

If $\mathcal{A}(t)=\bigl(a_1(t),\cdots,a_n(t)\bigr)$ is a {\it frame} for $\ell(t)$, i.e. $a_1(t),\cdots,a_n(t)$ are smooth
curves on $V$ spanning $\ell(t)$, then the condition of being fanning is equivalent to requiring that 
\begin{equation}\nonumber
\bigl(\mathcal{A}(t),\dot{\mathcal{A}}(t)\bigr)=\bigl(a_1(t),\cdots,a_n(t),\dot{a}_1(t),\cdots,\dot{a}_n(t)\bigr)
\end{equation}
\noindent be a frame for $V$. In general, we shall call a smooth curve $a(t)$ on $V$ satisfying $a(t)\in\ell(t)$ for all $t$ a {\it section} of $\ell(t)$.
The following definition 
does not depend on the choice of frame for $\ell(t)$.

\begin{definition}
The {\it fundamental endomorphism} of the fanning curve $\ell(t)$ is the curve ${\bf F}(t)\in{\rm End}(V)$ defined in the basis $\bigl(a_1(t),\cdots,a_n(t),\dot{a}_1(t),\cdots,\dot{a}_n(t)\bigr)$ by
\begin{equation}\nonumber
{\bf F}(t)a_i(t)=0~~~,~~~{\bf F}(t)\dot{a}_i(t)=a_i(t).
\end{equation}
\end{definition}

\begin{remark}
It is customary to abbreviate the notation in situations like the one above by 
${\bf F}(t)\mathcal{A}(t)={\rm O}~~~,~~~{\bf F}(t)\dot{\mathcal{A}}(t)=\mathcal{A}(t)$.
\end{remark}

The main thrust of \cite{AD1} is that the geometry of fanning curves under the action of ${\rm GL}(V)$ is completely described by ${\bf F}(t)$ and its 
derivatives $\dot{\bf F}(t),\ddot{\bf F}(t)$.

\subsubsection{The horizontal curve and the horizontal derivative}

The derivative $\dot{{\bf F}}(t)$
is a curve of reflections whose -1 eigenspace is $\ell(t)$. The 1-eigenspaces at each $t$ form thus a curve $h(t)$ on ${\rm Gr}_n(V)$, called
the {\it horizontal curve} of $\ell(t)$. The projection operators corresponding to the decomposition
\begin{equation}\label{decomposition}
V=\ell(t)\oplus h(t)
\end{equation}
are denoted by ${\bf P}_h(t)=\frac{1}{2}({\bf I}+\dot{\bf F}(t))$,  ${\bf P}_\ell(t)={\bf I}-{\bf P}_h(t)$.

\begin{definition}
The {\it horizontal derivative} at time $t=\tau$ is the isomorphism 
\begin{equation}\label{horizontalderivative}
{\bf H}(\tau):\ell(\tau)\rightarrow h(\tau)~,~~{\bf H}(\tau)v={\bf P}_h(\tau)\dot{a}(\tau),
\end{equation}
for $a:I\rightarrow V$ any section of $\ell(t)$ with $a(\tau)=v$. The horizontal derivative of a frame $\mathcal{A}(t)$
for $\ell(t)$ is thus a frame for $h(t)$, denoted by 
\begin{equation}\nonumber
\mathcal{H}(t)={\bf H}(t)\mathcal{A}(t).
\end{equation}
\end{definition}

We remark that the inverse of (\ref{horizontalderivative}) is the restriction of ${\bf F}(t)$ to $h(t)$,
\begin{equation}\label{inversehorizontalderivative}
{\bf H}(t)^{-1} = {\bf F}(t)|_{h(t)} : h(t) \rightarrow \ell(t). 
\end{equation}

Given a frame $\mathcal{A}(t)$ for $\ell(t)$, the fanning condition implies that there exist curves of $n\times n$ matrices $P(t)$ and $Q(t)$ such that
\begin{equation}\label{PandQ}
\ddot{A}(t)+\dot{\mathcal{A}}(t)P(t)+\mathcal{A}(t)Q(t)={\rm O}. 
\end{equation}
\noindent The frame is called {\it normal} if $P=0$, which in turn is equivalent to $\mathcal{H}(t)=\dot{\mathcal{A}}(t)$.

\subsubsection{The Jacobi endomorphism and the Schwarzian}

Since $\dot{\bf F}(t)$ is a curve of reflections, its derivative $\ddot{\bf F}(t)$ interchanges the decomposition (\ref{decomposition}). 
The {\it Jacobi endomorphism} of $\ell(t)$ is the curve ${\bf K}(t)$ on ${\rm End}(V)$ defined by
\begin{equation}\nonumber
{\bf K}(t)=\frac{1}{4}\ddot{{\bf F}}(t)^2=\dot{\bf P}_\ell(t)^2. 
\end{equation}
A nice description of ${\bf K}(t)$ is given in terms of the {\it Schwarzian} $\{\mathcal{A}(t),t\}$ of a frame $\mathcal{A}(t)$. If $P(t)$ and $Q(t)$ are
as in (\ref{PandQ}), then $\{\mathcal{A}(t),t\}$ is defined by 
\begin{equation}\label{schwarziandefinition}
\{\mathcal{A}(t),t\} = 2Q(t)-(1/2)P(t)^2-\dot{P}(t).
\end{equation}
Note that if $\mathcal{A}(t)$ is normal, then 
$$\ddot{A}(t)=-(1/2)\mathcal{A}(t)\{\mathcal{A}(t),t\}.$$

\begin{proposition}\label{propjacobiendo}
Given a frame $\mathcal{A}(t)$ for $\ell(t)$, the matrices of $(1/2)\ddot{{\bf F}}(t)=-\dot{\bf P}_\ell(t)$ and ${\bf K}(t)$ in the basis $(\mathcal{A}(t),\mathcal{H}(t))$ 
are, respectively, 
\begin{equation}\nonumber
\left(
\begin{array}{cc}
{\rm O} & -(1/2)\{\mathcal{A}(t),t\} \\
-{\bf I} & {\rm O}
\end{array}
\right)~~{\rm ,}~~
\left(
\begin{array}{cc}
 (1/2)\{\mathcal{A}(t),t\} & {\rm O} \\
 {\rm O} & (1/2)\{\mathcal{A}(t),t\}
\end{array}
\right).
\end{equation}
\end{proposition}

\subsubsection{Fanning curves of Lagrangian subspaces}

Let us now suppose that $V$ is endowed with a symplectic form $\omega$. Recall that a subspace $\ell\subseteq V$ is called 
{\it Lagrangian} if $\ell=\ell^\omega:=\{u\in V~:~\omega(u,v)=0~\mbox{for all $v\in \ell$}\}$, and the collection of all such subspaces 
forms a submanifold $\Lambda(V,\omega)$, or simply $\Lambda(V)$, of ${\rm Gr}_n(V)$, the so-called {\it Lagrangian Grassmannian} of $V$. 
For each $\ell\in\Lambda(V)$ there is a canonical identification
\begin{equation}\label{identificationwronskian}
T_\ell\Lambda(V)\cong{\rm Bil}_{\rm sym}(\ell),
\end{equation}
through which the velocity vectors of a smooth curve
$\ell:I\subseteq\mathds{R}\rightarrow\Lambda(V)$
are regarded as symmetric bilinear forms. Concretely,
\begin{definition}
The {\it Wronskian} at time $t=\tau$ of a smooth curve $\ell:I\subseteq\mathds{R}\rightarrow\Lambda(V)$ is the symmetric
bilinear form $W(\tau)\in{\rm Bil}_{\rm sym}(\ell(\tau))$ given by
$W(\tau)(u,v)=\omega(u,\dot{a}(\tau))$, 
for $a:I\rightarrow V$ any section of $\ell(t)$ with $a(\tau)=v$.
\end{definition}

In this setting, the condition for a curve $\ell:I\subseteq\mathds{R}\rightarrow\Lambda(V)$ to be fanning is equivalent to 
$W(t)$ being non-degenerate for all $t$. Furthermore,

\begin{proposition}\label{propfanninglagr}
For a fanning curve $\ell(t)$ on $\Lambda(V)$, the following hold:
\begin{enumerate}
\item The fundamental endomorphism ${\bf F}(t)$ takes values in the Lie algebra $\mathfrak{sp}(V)$.
\item The horizontal curve $h(t)$ consists of Lagrangian subspaces.
\item The restriction of ${\bf K}(t)$ to $\ell(t)$ is symmetric with respect to $W(t)$.
\end{enumerate}
\end{proposition}

\subsubsection{Transformation properties}

Fanning curves on ${\rm Gr}_n(V)$, resp. $\Lambda(V)$, are naturally acted upon by ${\rm GL}(V)$, resp. ${\rm SP}(V)$, and by the group ${\rm Diff}(\mathds{R})$ of diffeomorphisms
of $\mathds{R}$ via reparametrization.

\begin{proposition}\label{transformationproperties}
Let $\ell(t)$ be a fanning curve on ${\rm Gr}_n(V)$, resp. $\Lambda(V)$. Given ${\bf T}\in{\rm GL}(V)$, resp. ${\rm SP}(V)$, and 
$s\in{\rm Diff}(\mathds{R})$, then
\begin{enumerate}
\item The fundamental endomorphism, the Wronskian, and the Jacobi endomorphism of ${\bf T}\ell(t)$ are, respectively,
${\bf T}{\bf F}(t){\bf T}^{-1}$, $({\bf T}|_{\ell(t)})_*W(t)$ and ${\bf T}{\bf K}(t){\bf T}^{-1}$.
\item The fundamental endomorphism, the Wronskian, and the Jacobi endomorphism of $\ell(s(t))$ are, respectively,
$\dot{s}(t){\bf F}(s(t))$, $\dot{s}(t)W(s(t))$ and 
\begin{equation}\nonumber
\dot{s}(t)^2{\bf K}(s(t))+\frac{1}{2}\{s(t),t\}{\bf I},
\end{equation}
where $\{s(t),t\}=(d/dt)(\dot{s}^{-1}\ddot{s})-(1/2)(\dot{s}^{-1}\ddot{s})^2$ is the Schwarzian derivative of $s(t)$.
\end{enumerate}
\end{proposition}

\section{Moving planes, Jacobi curves and their invariants}\label{sectionmovingplanes}

Let us consider a moving plane $\mathscr{P}$ on a smooth manifold $X$, of the type
\begin{equation}\label{movingplane}
\mathscr{P}=(\Delta_{2n},\Delta_n,\Phi_t),
\end{equation}
and let
$S$ be the vector field on $X$ that generates $\Phi_t$. In particular, we will also be interested in the cases where
\begin{itemize}
\item[(I)] $X=(X^{2m},\omega)$ is a symplectic manifold, $\Delta_{2n}=TX$, $\Delta_n$ is a Lagrangian distribution
on $X$ (i.e. each $\Delta_n(x)$ is a Lagrangian subspace of $T_xX$), and $\Phi_t$ is a symplectic flow (i.e. $(\Phi_t)^*\omega=\omega$).
\item[(II)] $X=(X^{2m+1},\alpha)$ is an exact contact manifold, in which case we let $\omega={\rm d}\alpha$, 
$\Delta_{2n}$ is the contact distribution ${\rm ker}(\alpha)$, 
$\Delta_n$ is a Legendrian distribution $\mathcal{L}$ (i.e. each $\mathcal{L}_x$ is a Lagrangian subspace of $({\rm ker}(\alpha_x),\omega_x)$),
and $\Phi_t$ is an exact contact flow (i.e. $(\Phi_t)^*\alpha=\alpha$).
\end{itemize}
\noindent It then follows that the Jacobi curve $\ell_x(t)$ of $\mathscr{P}$, based at a given $x\in X$ (recall (\ref{jacobicurvedefinition}) ), is a curve in the half-Grassmannian ${\rm Gr}_n(\Delta_{2n}(x))$ and
that, in cases (I) and (II), $\ell_x(t)$ takes values on the Lagrangian Grassmannian 
$\Lambda(V)$, where $V=(\Delta_{2n}(x),\omega_x)$.

\begin{example}\label{examplesmovingplanes}
The examples to keep in mind are provided by the geodesic flows of sprays and Finsler metrics.
Let $S$ be a spray on $M^n$.
\begin{enumerate}
\item The action of $\Phi_t^S$ on the vertical distribution $\mathcal{V}TM$ gives rise to the moving plane 
\begin{equation}\label{movingplanespray}
\mathscr{P}=\bigl(T(TM\backslash 0),\mathcal{V}TM,\Phi_t^S\bigr).
\end{equation}
\item Suppose $S$ is the geodesic spray $S_F$ of a Finsler metric $F$. The canonical 1-form $\alpha$ pulls-back to the null form on each fiber of $\tau:T^*M\rightarrow M$,
and so does $\omega$. In particular, $\mathcal{V}T^*M$ and, hence, $\mathcal{V}TM$ are Lagrangian distributions on $(T^*M\backslash,\omega)$ and $(TM\backslash 0,\omega_F)$,
respectively. Furthemore, the flows $\Phi_t^{S_F}$ and $\Phi_t^{S_F^*}$ are symplectic. Therefore, (\ref{movingplanespray}) is of type (I) with respect to $\omega_F$, and
\begin{equation}\nonumber
 \mathscr{P}_*=\bigl(T(T^*M\backslash 0),\mathcal{V}T^*M,\Phi_t^{S_F^*}\bigr)
\end{equation}
is of type (I) on $(T^*M\backslash 0,\omega)$.

\item Still in the Finslerian setting, the tangent spaces to the fibers of 
$\Sigma_FM\rightarrow M$ and $\Sigma_F^*M\rightarrow M$ define, respectively, the {\it vertical distributions}
$\mathcal{V}\Sigma_FM$ and $\mathcal{V}\Sigma_F^*M$. As before, these are Legendrian distributions on $(\Sigma_FM,\alpha_F)$ and 
$(\Sigma_F^*M,\alpha)$. We therefore obtain moving planes of type (II) on these contact manifolds,
\begin{equation}\nonumber
\mathscr{P}^{\rm c}  =  \bigl( {\rm ker}(\alpha_F), \mathcal{V}\Sigma_FM, \Phi_t^{S_F} \bigr),
~~~\mathscr{P}_*^{\rm c}  =  \bigl({\rm ker}(\alpha),\mathcal{V}\Sigma_F^*M,\Phi_t^{S_F^*}\bigr). 
\end{equation}
\end{enumerate}
\end{example}

Given a frame $U_1,\cdots,U_n$ for $\Delta_n$ defined around a given point $x\in X$, a
frame for the corresponding Jacobi curve $\ell_x(t)$ is obtained by setting 
\begin{equation}\nonumber
a_i(t)=\bigl({\Phi_t}^*U_i\bigr)(x)~,~~i=1,\cdots,n. 
\end{equation}

\noindent From the properties of flows, the derivative $\dot{a}_i(t)$ computes as
\begin{equation}\label{derivativeframe}
\dot{a}_i(t) = \bigl({\Phi_t}^*[S,U_i]\bigr)(x),
\end{equation}
so that we conclude

\begin{lemma}
The Jacobi curve $\ell_x(t)$ is fanning if, and only if, along the flow line $t\mapsto\Phi_t(x)$, 
\begin{equation}\label{framemovingplane}
U_1,\cdots,U_n,[S,U_1],\cdots,[S,U_n] 
\end{equation}
constitute a frame for $\Delta_{2n}$. In particular, this condition on $(\ref{framemovingplane})$ does not depend on the choice
of the local frame $U_1,\cdots,U_n$.
\end{lemma}

\begin{definition}
We shall call the moving plane $\mathscr{P}$ {\it regular} if (\ref{framemovingplane}) are a local frame for $\Delta_{2n}$ whenever 
$U_1,\cdots,U_n$ are a local frame for $\Delta_n$.
\end{definition}

For a regular moving plane (\ref{movingplane}), we shall denote by ${\bf F}_x(t)$, ${\bf P}_{\ell_x}(t)$, ${\bf K}_x(t)$, $h_x(t)$ and, in cases (I) and (II), 
$W_x(t)$ the invariants of the fanning curve $\ell_x(t)$, for $x\in X$. Evaluating at $t=0$ and by varying $x$, 
one thus obtains, respectively, sections $\mathcal{F}$, $\mathcal{P}_{\Delta_n}$,
$\mathcal{K}$ of ${\rm End}(\Delta_{2n})\rightarrow X$, a distribution $\mathcal{H}\subset\Delta_{2n}$ and a section $\mathcal{W}$ of ${\rm Bil}_{\rm sym}(\Delta_n)\rightarrow X$.

\begin{lemma}
Along an orbit $t\mapsto\Phi_t(x)$, $\mathcal{H}_{\Phi_t(x)}$, $\mathcal{F}_{\Phi_t(x)}$, $(\mathcal{P}_{\Delta_n})_{\Phi_t(x)}$, $\mathcal{K}_{\Phi_t(x)}$ and $\mathcal{W}_{\Phi_t(x)}$
correspond to $h_x(t)$, ${\bf F}_x(t)$, ${\bf P}_{\ell_x}(t)$, ${\bf K}_x(t)$ and $W_x(t)$ via the isomorphisms
\begin{equation}\nonumber
{\rm d}\Phi_t(x)|_{\Delta_{2n}(x)}:\Delta_{2n}(x)  \rightarrow  \Delta_{2n}(\Phi_t(x)), ~~~
{\rm d}\Phi_t(x)|_{\ell_x(t)}:\ell_x(t)  \rightarrow  \Delta_n(\Phi_t(x)).  
\end{equation}  
\end{lemma}
\begin{proof}
Just note that ${\rm d}\Phi_t(x)\ell_x(s)=\ell_{\Phi_t(x)}(s-t)$ and apply Proposition \ref{transformationproperties}. 
\end{proof}

\noindent {\it Reduction by a contact type hypersurface.}
Let $X^{2n}$, $\omega$, $S$, $\Sigma^{2n-1}$, $C$, $\alpha$, be as in Remark \ref{remarkcontacttype}. Let, furthermore, $\Delta_n$ be a Lagrangian distribution on $X$ such 
that $\Delta_n\subset{\rm ker}(\alpha)$ and $C\in\Delta_n$, so that $L:=\Delta_n\cap T\Sigma$
is a Legendrian distribution on $\Sigma$. Then,

\begin{proposition}\label{propcontactrelation}
Given $x\in\Sigma$, let $\ell_x(t)\in\Lambda(T_xX)$ and $\ell_x^{\rm c}(t)\in\Lambda({\rm ker}(\alpha)_x)$ be the Jacobi curves of the 
moving planes $\mathscr{P}=(TX,\Delta_n,\Phi_t^S)$ and
$\mathscr{P}^{\rm c}=({\rm ker}(\alpha),L,\Phi_t^S|_\Sigma)$, on $X$ and $\Sigma$ respectively, based at $x$, 
and let $W_x(t)$ and $W_x^{\rm c}(t)$ be their Wronskians. Then,
\begin{enumerate}
\item We have a $W_x(t)$-orthogonal decomposition
\begin{equation}\label{contactdecomposition}
\ell_x(t)=\ell_x^{\rm c}(t)\oplus{\rm span}[C_x-tS_x],
\end{equation}
and the restriction of $W_x(t)$ to $\ell_x^{\rm c}(t)$ is equal to $W_x^{\rm c}(t)$.
\item $\mathscr{P}$ is regular in a neighborhood of $\Sigma$ if, and only if, $\mathscr{P}^{\rm c}$ is regular. This being the case, the horizontal curves
$h_x(t)$, $h_x^{\rm c}(t)$, and the Jacobi endomorphisms ${\bf K}_x(t)$, ${\bf K}_x^{\rm c}(t)$, of $\ell_x(t)$ and $\ell_x^{\rm c}(t)$, satisfy
\begin{equation}\nonumber
h_x^{\rm c}(t)  =  h_x(t)\cap {\rm ker}(\alpha_x)~,~~{\bf K}_x(t)\big|_{L_x}  =  {\bf K}_x^{\rm c}(t)\big|_{L_x}~,~~{\bf K}_x(t)(C_x-tS_x)  =  0.
\end{equation}
\end{enumerate}
\end{proposition}

\begin{proof}
By hypothesis, we can choose a local frame for $\Delta_n$ around $x$, $U_1,\cdots,U_n$, such that $U_n=C$ and that, along $\Sigma$, $U_1,\cdots,U_{n-1}$ is
a frame for $L$. Let $\mathcal{A}^{\rm c}(t)$ and $\mathcal{A}(t)=(\mathcal{A}^{\rm c}(t),a_n(t))$ be the corresponding frames for $\ell_x^{\rm c}(t)$ and
$\ell_x(t)$, respectively. It follows from $[C,S]=S$ and ${\Phi_t}^*S=S$ that
\begin{equation}\nonumber
\dot{a}_n(t)  =  ({\rm d}/{\rm d}t)({\Phi_t}^*C)_x =  ({\Phi_t}^*[S,C])_x =  - S_x. 
\end{equation}
Since $a_n(0)=C_x$, we obtain $a_n(t)=C_x-tS_x$ and (\ref{contactdecomposition}) follows. Observe that we have a direct sum decomposition 
$T_x\Sigma={\rm ker}(\alpha_x)\oplus{\rm span}[S_x]\oplus{\rm span}[C_x]$. Since ${\rm span}(\mathcal{A}^{\rm c}(t),\dot{\mathcal{A}}^{\rm c}(t))\subseteq{\rm ker}(\alpha_x)$,
and $a_n(t)=C_x-tS_x$, it thus follows that 
$\ell_x(t)$ is fanning if, and only if, $\ell_x^{\rm c}(t)$ is fanning. Being the case,
let $P(t)$, $Q(t)$, and $P^{\rm c}(t)$, $Q^{\rm c}(t)$ be given by (\ref{PandQ}) with respect to $\mathcal{A}(t)$ and $\mathcal{A}^{\rm c}(t)$, respectively. Since 
$\ddot{a}=0$, it follows that $P={\rm diag}(P^{\rm c},0)$ and $Q={\rm diag}(Q^{\rm c},0)$. Recalling (\ref{schwarziandefinition}), we conclude 
that $\{\mathcal{A}(t),t\}={\rm diag}(\{\mathcal{A}^{\rm c}(t),t\},0)$. The assertion about the Jacobi endomorphisms follows now from Proposition \ref{propjacobiendo}.
The ones about the Wronskians and the horizontal curves are analogues.
\end{proof}

\subsection{Expressions in terms of Lie brackets}\label{sectionliebrackets}

The objects $\mathcal{F}$, $\mathcal{H}$, $\mathcal{K}$ and $\mathcal{W}$ can be described in terms of taking Lie brackets with the vector field $S$.
Firstly, if $\mathcal{T}$ is a section of ${\rm End}(\Delta_{2n})\rightarrow X$, 
the Lie derivative $[S,\mathcal{T}]$ is defined and it holds that
\begin{equation}\nonumber
\frac{\rm d}{{\rm d}t}(\Phi_t)^*\mathcal{T}=(\Phi_t)^*[S,\mathcal{T}]. 
\end{equation}
It follows from this, (\ref{derivativeframe}), and $\S$\ref{fanningcurvessubsection} that
\begin{enumerate}
\item The endomorphism $\mathcal{F}$ is characterized by 
\begin{equation}\label{fundamentalendomorphismmovingplane}
\mathcal{F}(U_i)=0,~~\mathcal{F}([S,U_i])=U_i, 
\end{equation}
$i=1,\cdots,n$,
for any local frame $U_1,\cdots,U_n$ for $\Delta_n$.
\item The Lie derivative $[S,\mathcal{F}]$ is a section of reflections across $\mathcal{H}$.
\item $\mathcal{K}$ is the square of $(1/2)[S,[S,\mathcal{F}]]=-[S,\mathcal{P}_{\Delta_n}]$. 
Furthermore, let {\rm H} be the section of ${\rm Iso}(\Delta_n,\mathcal{H})\rightarrow X$ corresponding to (\ref{horizontalderivative}), so that 
\begin{equation}\nonumber
{\rm H}(U)=\mathcal{P}_\mathcal{H}([S,U]),
\end{equation}
for $U$ a vector field tangent to $\Delta_n$. Then,
\begin{equation}\label{jacobi1}
\mathcal{K}|_{\Delta_n}  =  [S,\mathcal{P}_{\Delta_n}]\big|_\mathcal{H}\circ{\rm H}. 
\end{equation}
Applying (\ref{jacobi1}) to a vector field $U$ tangent to $\Delta_n$ and using that $\mathcal{P}_{\Delta_n}$ vanishes on $\mathcal{H}$, one obtains 
\begin{equation}\label{jacobi2}
\mathcal{K}(U) = -\mathcal{P}_{\Delta_n}\bigl([S,{\rm H}(U)]\bigr). 
\end{equation}

\item In cases (I) and (II), given vector fields $U,V$ tangent to $\Delta_n$, then
\begin{equation}\label{wronskkian2}
\mathcal{W}(U,V)=\omega(U,[S,V]). 
\end{equation}
\end{enumerate}

\subsection{The Jacobi curves associated to sprays and Finsler metrics}\label{subsectionjacobisprays}

Let us now come back to the moving planes from Example \ref{examplesmovingplanes}. Throughout this section, let $S$ be fixed
a spray on $M^n$.

\begin{lemma}
The moving plane $(\ref{movingplanespray})$ is regular. 
\end{lemma}
\begin{proof}
Let $X_1,\cdots,X_n$ be a local frame for $\mathcal{V}TM$. Since the almost-tangent structure $\mathcal{J}$ satisfies (cf. Lemma \ref{lemmaalmosttangent}) 
\begin{equation}\label{almosttangentstructureequation}
\mathcal{J}(X_i)=0~,~~\mathcal{J}([S,X_i])=-X_i,
\end{equation}
a linear dependence relation among $X_1,\cdots,X_n,[S,X_1],\cdots,[S,X_n]$ would give a linear dependence relation among $X_1,\cdots,X_n$.
\end{proof}
Let, therefore, $\mathcal{F}$, $\mathcal{H}$, $\mathcal{K}$ be the corresponding differential invariants of $\mathscr{P}$. From 
(\ref{almosttangentstructureequation}) and (\ref{fundamentalendomorphismmovingplane}) we obtain
\begin{equation}\label{fundamentalendoalmosttangent}
\mathcal{F}=-\mathcal{J}. 
\end{equation}
In particular, since $[S,\mathcal{F}]$ consists of reflections across $\mathcal{H}$, we recover the following result \cite[Prop. I.41]{Grifone}.

\begin{corollary}\label{corollaryfundamentalendomorphism}
The Lie derivative $\Gamma_S:=-[S,\mathcal{J}]$ is a section of reflections of ${\rm End}(T(TM\backslash 0))\rightarrow TM\backslash 0$ such that
${\rm ker}(\Gamma_S+{\bf I})=\mathcal{V}TM$. 
\end{corollary}

The section $\Gamma_S$ is an example of a {\it connection} on $M$ in the sense of Grifone (cf. \cite[Def. I.14]{Grifone}); indeed, $\Gamma_S$ is the canonical connection
associated to the spray $S$. The corresponding Ehresmann connection on $TM\backslash 0$, given by the 1-eigenspaces of $\Gamma_S$, is the so-called
{\it horizontal tangent bundle} (associated
to $S$), $\mathcal{H}TM={\rm ker}([S,\mathcal{J}]-{\bf I})$, so that
\begin{equation}\label{horizontaltangentspacedecomposition}
T(TM\backslash 0) = \mathcal{H}TM\oplus\mathcal{V}TM. 
\end{equation}
Therefore, we have recovered $\mathcal{H}TM$ as the horizontal distribution $\mathcal{H}$ of $\mathscr{P}$,
\begin{equation}
\mathcal{H}TM=\mathcal{H}, 
\end{equation}
and we can unambiguously denote by $\mathcal{P}_\mathcal{H}$ and $\mathcal{P}_\mathcal{V}$ the projections relative to (\ref{horizontaltangentspacedecomposition}).
Note that the homogeneity $[C,S]=S$ of $S$ implies that $S$ is tangent to $\mathcal{H}TM$.

\vskip 5pt

\noindent In terms of Jacobi curves: fixing a non-zero vector $v\in T_mM$, let $\gamma:I\subseteq\mathds{R}\rightarrow M$ be the geodesic of $S$ with $\dot{\gamma}(0)=v$, 
and let
\begin{equation}\nonumber
\ell_v:I\subseteq\mathds{R}\rightarrow {\rm Gr}_n(T_vTM) 
\end{equation}
be the Jacobi curve of $\mathscr{P}$ based at $v$. We have shown that

\begin{proposition}
Under the isomorphism ${\rm d}\Phi_t^S:T_vTM\rightarrow T_{\dot{\gamma}(t)}M$, the endomorphism $-\mathcal{J}_{\dot{\gamma}(t)}$ corresponds to ${\bf F}_v(t)$ and, thus, 
$(\Gamma_S)_{\dot{\gamma}(t)}$ corresponds to $\dot{\bf F}_v(t)$. Therefore, $\mathcal{H}_{\dot{\gamma}(t)}TM={\rm d}\Phi_t^S(v)h_v(t)$.
\end{proposition}

Next we show how the notions of covariant derivative and curvature endomorphism along $\gamma$, associated to $S$, can be recovered in this setting. We refer the reader to
$\S$\ref{sectionconnectionsandcurvature} for the definitions of those concepts as well as for the proofs of the following results.
\par Consider, for each $t$, the isomorphism
\begin{equation}\label{isomorphism3}
\iota_{v,t}:={i_{\dot{\gamma}(t)}}^{-1}\circ {\rm d}\hskip 1pt \Phi_t^{S}(v)~:~\ell_v(t)\longrightarrow T_{\gamma(t)}M. 
\end{equation}
\noindent For $t=0$ this is just the tautological isomorphism $i_v:\ell_v(0)=\mathcal{V}_vTM\rightarrow T_mM$.

\begin{proposition}\label{propjacobicurvatureendo}
The endomorphisms ${\bf K}_v(t)|_{\ell_v(t)}:\ell_v(t)\rightarrow\ell_v(t)$ and ${\bf R}_{\dot{\gamma}(t)}:T_{\gamma(t)}M\rightarrow T_{\gamma(t)}M$ correspond
under $(\ref{isomorphism3})$. 
\end{proposition}

It therefore follows from Proposition \ref{propjacobiendo} that, given a frame $V_1,\cdots,V_n\in\mathfrak{X}(\gamma)$, if $\mathcal{A}(t)$ is the corresponding frame for $\ell_v(t)$, then
the matrix of ${\bf R}_{\dot{\gamma}(t)}$ with respect to that frame is $(1/2)\{\mathcal{A}(t),t\}$.

\begin{proposition}\label{propdynamicalcovariantderivative}
Given $V\in\mathfrak{X}(\gamma)$, let $a(t)\in\ell_v(t)$ correspond to $V$ via $(\ref{isomorphism3})$. Then ${\rm D}^{\dot{\gamma}}V/{\rm d}t$ corresponds to
${\bf P}_{\ell_v}(t)\dot{a}(t)$ via $(\ref{isomorphism3})$.
\end{proposition}

\subsubsection{The case of a Finsler metric}\label{sectionjacobicurvesfinsler}

Let us now suppose that $S$ is the geodesic spray of a Finsler metric $F$ on $M$.
\par In this case, $\ell_v(t)$ takes values in $\Lambda(T_vTM)$ if we regard
$\mathscr{P}$ as of type (I) with respect to $\omega_F$.

\begin{proposition}\label{propwronskianfundamentaltensor}
The Wronskian $W_v(t)$ of $\ell_v(t)$ corresponds, under $(\ref{isomorphism3})$, to the fundamental tensor $g_F(\dot{\gamma}(t))$ of $F$ at $\dot{\gamma}(t)$.
\end{proposition}
\begin{proof}
This is equivalent to show that, given vector fields $U$, $V$ on $M$, then \newline 
$\mathcal{W}(V^\mathfrak{v},U^\mathfrak{v})(w)=g_F(w)(V,U)$, for $\mathcal{W}$
the section of ${\rm Bil}_{\rm sym}(\mathcal{V}TM)\rightarrow TM\backslash 0$ associated to $\mathscr{P}$.
On one hand, from (\ref{wronskkian2}) 
\begin{eqnarray}
\mathcal{W}(V^\mathfrak{v},U^\mathfrak{v})  & = & -{\rm d}\alpha_F\bigl(V^\mathfrak{v},[S,U^\mathfrak{v}]\bigr)\nonumber\\
& = & [S,U^\mathfrak{v}]\bigl(\alpha_F(V^\mathfrak{v})\bigr)-V^\mathfrak{v}\bigl(\alpha_F([S,U^\mathfrak{v}])\bigr)
-\alpha_F\bigl(\bigl[[S,U^\mathfrak{v}],V^\mathfrak{v}\bigr]\bigr).\nonumber
\end{eqnarray}
On the other hand, Lemma \ref{lemmaalmosttangent} implies that $[S,U^\mathfrak{v}]$ is $\pi$-related to $-U$ (i.e., $d\pi[S,U^\mathfrak{v}] = -U$). From this it follows that 
$[[S,U^\mathfrak{v}],V^\mathfrak{v}]$ is vertical and that $\alpha_F([S,U^\mathfrak{v}])$ is the function $u\mapsto-\mathscr{L}_F(u)U$.
Since $\alpha_F$ vanishes on vertical vectors and ${\rm d}_f\mathscr{L}_F(w)V=g_F(w)(V,\cdot)$ we therefore obtain 
$\mathcal{W}(U^\mathfrak{v},V^\mathfrak{v})(w)=-V^\mathfrak{v}\bigl(\alpha_F([S,U^\mathfrak{v}])\bigr)(w)=g_F(w)(V,U)$.
\end{proof}

As a corollary of this and Proposition \ref{propjacobicurvatureendo}, we get the flag curvature in terms of the Jacobi curve:

\begin{theorem}\label{theoremjacobiflag}
Given a $2$-plane $\Pi={\rm span}[v,u]$ in $T_{\gamma(t)}M$, with $g_F(v)(v,u)=0$, let
$a\in\ell(0)=\mathcal{V}_vTM$ be $a=i_v(u)$. Then,
\begin{equation}\nonumber
K_F(v,\Pi) = \frac{1}{F(v)^2}\frac{W_v(0)\bigl({\bf K}_v(0)a,a\bigr)}{W_v(0)(a,a)} 
\end{equation}
\end{theorem}

\noindent{\it The co-tangent setting.}
Let $\xi=\mathscr{L}_F(v)$ and let $\ell_\xi(t)\in\Lambda(T_\xi T^*M)$ be the Jacobi curve of $\mathscr{P}_*$ based at $\xi$.
With the help of the Legendre transformation $\mathscr{L}_F$, one obtains an isomorphism
\begin{equation}\label{isomorphism4}
\iota_{\xi,t}:=\bigl({\rm d}_f\mathscr{L}_F(\dot{\gamma}(t))\bigr)^{-1} \circ \bigl(i_{\mathscr{L}_F(\dot{\gamma}(t))}\bigr)^{-1}\circ 
{\rm d}\Phi_t^{S_F^*}(\xi)~:~ \ell_\xi(t)\longrightarrow T_{\gamma(t)}M. 
\end{equation}
Note from (\ref{fiberderivative}) that, for $t=0$, (\ref{isomorphism4}) is the inverse of
\begin{equation}\nonumber
T_mM\longrightarrow\ell_\xi(0)=\mathcal{V}_\xi T^*M,~~ w\mapsto i_\xi\bigl( g_F(v)(w,\hskip 1pt\cdot\hskip 1pt)\bigr).
\end{equation}
Now, since $\mathscr{L}_F$ is a symplectic diffeomorphism that maps the data in $\mathscr{P}$ to the ones in $\mathscr{P}_*$, then
\begin{equation}\label{isomorphismlegendre}
{\rm d}\mathscr{L}_F(v): T_vTM \rightarrow T_\xi T^*M 
\end{equation}
is a symplectic isomorphism mapping $\ell_v(t)$ to $\ell_\xi(t)$. In particular, it follows from Proposition \ref{transformationproperties} that
${\bf F}_v(t)$, $W_v(t)$, $h_v(t)$, ${\bf K}_v(t)$, correspond to ${\bf F}_\xi(t)$, $W_\xi(t)$, $h_\xi(t)$, ${\bf K}_\xi(t)$, 
under (\ref{isomorphismlegendre}). Therefore, $W_\xi(t)$ and ${\bf K}_\xi(t)|_{\ell_\xi(t)}$ correspond to $g_F(\dot{\gamma}(t))$ and ${\bf R}_{\dot{\gamma}(t)}$, respectively,
under (\ref{isomorphism4}).

\vskip 6pt

\noindent{\it The contact setting.} Suppose $F(v)=1$, hence $F^*(\xi)=1$, and let 
\begin{equation}\nonumber
\ell_v^{\rm c}(t)\in\Lambda\bigl({\rm ker}(\alpha_F)_v\bigr),~~~
\ell_\xi^{\rm c}(t)\in\Lambda\bigl({\rm ker}(\alpha)_\xi\bigr) 
\end{equation}
be the Jacobi curves of $\mathscr{P}^{\rm c}$ and $\mathscr{P}_*^{\rm c}$, based at $v$ and $\xi$, respectively.
Observe that $\mathscr{P}$ and $\mathscr{P}^{\rm c}$, as well as $\mathscr{P}_*$ and $\mathscr{P}_*^{\rm c}$, 
fit within the setting in Proposition \ref{propcontactrelation}. Therefore, since (\ref{tautologicalisomorphism}) maps
${\rm ker}\hskip 1pt g_F(w)(w,\cdot)$ onto $\mathcal{V}_w\Sigma_FM$ (for $w\in\Sigma_FM$), then (\ref{isomorphism3}) and (\ref{isomorphism4}) 
restrict to isomorphisms
\begin{eqnarray}
\iota_{v,t}|_{\ell_v^{\rm c}(t)}~:~\ell_v^{\rm c}(t) & \rightarrow & {\rm ker}\hskip 1pt g_F(\dot{\gamma}(t))(\dot{\gamma}(t)\hskip 1pt,\hskip 1pt\cdot) \label{isomorphism5}\\
\iota_{\xi,t}|_{\ell_\xi^{\rm c}(t)}~:~\ell_\xi^{\rm c}(t) & \rightarrow & {\rm ker}\hskip 1pt g_F(\dot{\gamma}(t))(\dot{\gamma}(t)\hskip 1pt, \hskip 1pt\cdot) \nonumber
\end{eqnarray}
under which $W_v^{\rm c}(t)$, $W_\xi^{\rm c}(t)$, and ${\bf K}_v^{\rm c}(t)$, ${\bf K}_\xi^{\rm c}(t)$, respectively, correspond to the restrictions of $g_F(\dot{\gamma}(t))$ and
${\bf R}_{\dot{\gamma}(t)}$ to ${\rm ker}\hskip 1pt g_F(\dot{\gamma}(t))(\dot{\gamma}(t)\hskip 1pt,\hskip 1pt\cdot)$. 
Also, the horizontal curves $h_v^{\rm c}(t)$, for $v\in\Sigma_FM$, give rise to the standard horizontal distribution on $\Sigma_FM$,
\begin{equation}\label{horizontalcontact}
\mathcal{H}\Sigma_FM = \mathcal{H}TM \cap {\rm ker}(\alpha_F). 
\end{equation}

\subsection{Invariants from the connections point of view}\label{sectionconnectionsandcurvature}

The linear connections arising in the theory of sprays and Finsler metrics are naturally defined on the vertical tangent bundle
(\ref{verticalbundle}). As shown in \cite{Henrique2}, the classical connections of Berwald, Cartan, Chern and Rund, and Hashiguchi are examples
of linear connections $\nabla$ on (\ref{verticalbundle}) satisfying the following two conditions (recall from Corollary \ref{corollaryfundamentalendomorphism}
the definition of $\Gamma_S$)
\begin{itemize}
\item[\bf L.] $\nabla$ is {\it lift} of the connection $\Gamma_S$, i.e. given $X\in T(TM\backslash 0)$, then
\begin{equation}
\nabla_XC = \mathcal{P}_\mathcal{V}(X).
\end{equation}
\item[\bf T.] ${\rm T}(S,X)=0$ for all $X\in T(TM\backslash 0)$; here, the {\it torsion} T of $\nabla$ is the $\mathcal{V}TM$-valued tensor field on $TM\backslash 0$
defined (in terms of vector fields) by
\begin{equation}
{\rm T}(X,Y) = \nabla_X\mathcal{J}(Y)-\nabla_Y\mathcal{J}(X)-\mathcal{J}([X,Y]). 
\end{equation}
\end{itemize}

On the other hand, the above conditions on a linear connection $\nabla$ guarantee that the covariant derivatives and the curvature endomorphism on $M$ induced by $\nabla$,
as defined next, are intrinsic to the spray $S$.

\subsubsection{The covariant derivative, the curvature endomorphism and the flag curvature}\label{subsubcovcurv}

Throughout this section, let  $\nabla$ be fixed a connection on (\ref{verticalbundle}) satisfying {\bf L.} and {\bf T.}. For a smooth curve 
$\gamma:I\subseteq\mathds{R}\rightarrow M$, we let $\mathfrak{X}(\gamma)$ denote the space of vector fields along $\gamma$.

\begin{definition}
Given a smooth curve $\gamma:I\subseteq\mathds{R}\rightarrow M$, with $\gamma(t_0)=m$, and non-null vector $w\in T_mM$, 
the map ${\rm D}^w/{\rm d}t:\mathfrak{X}(\gamma)\rightarrow T_mM$ is defined by
\begin{equation}\nonumber
\frac{{\rm D}^w V}{{\rm d}t} = {i_w}^{-1}\Bigl(\frac{\nabla V^\mathfrak{v}}{{\rm d}t}\Bigr)(t_0),
\end{equation}
where $V^\mathfrak{v}$ is the vertical lift of $V$ along the {\it horizontal lift} $\overline{\gamma}:I'\subseteq\mathds{R}\rightarrow TM\backslash 0$ of $\gamma$ through
$w$ at $t=t_0$ (i.e. $\overline{\gamma}$ is the lift of $\gamma$ that is tangent to $\mathcal{H}TM$ and $\overline{\gamma}(t_0)=w$).
\end{definition}

By considering a nowhere null vector field $W\in\mathfrak{X}(\gamma)$, one thus obtains a map 
${\rm D}^W/{\rm d}t:\mathfrak{X}(\gamma)\rightarrow\mathfrak{X}(\gamma)$
that satisfies the properties of a covariant derivative.  
\begin{propdef}
If $\gamma$ is a regular curve, then the map 
\begin{equation}
{\rm D}^{\dot{\gamma}}/{\rm d}t:\mathfrak{X}(\gamma)\rightarrow\mathfrak{X}(\gamma)
\end{equation}
does not depend on the 
choice of $\nabla$, but only on $S$. This is the covariant derivative along $\gamma$ associated to $S$.
\end{propdef}
By using vertical and horizontal lift operations one can bring the curvature tensor of $\nabla$,
\begin{equation}\nonumber
\mathcal{R}(X,Y)Z = \nabla_{[X,Y]}Z - [\nabla_X,\nabla_Y]Z, 
\end{equation}
down to $M$ so as to define, for given $m\in M$ and non-null vector $w\in T_mM$, a tri-linear map $R_w:T_mM\times T_mM\times T_mM\rightarrow T_mM$ by
\begin{equation}\label{curvaturetensor}
R_w(u,v)z = {i_w}^{-1}\mathcal{R}(u^\mathfrak{h},v^\mathfrak{h})z^\mathfrak{v}, 
\end{equation}
where the vertical and horizontal lifts are at $w$. The following is a consequence of Proposition \ref{propjacobicurvatureendo} which we shall prove in $\S$\ref{sectionproofs}. 
\begin{propdef}
The endomorphism ${\bf R}_w:T_mM\rightarrow T_m M$ defined by ${\bf R}_w(v)=R_w(w, v)w$ does not depend on the choice of $\nabla$, but only on $S$. 
This is the curvature endomorphism of $S$ in the direction $w$.
\end{propdef}

Let us now suppose that $S$ is the geodesic spray of a Finsler metric $F$ on $M$. 

\begin{propdef}
The curvature endomorphism ${\bf R}_w$ is symmetric with respect to $g_F(w)$. As a consequence, given a $2$-dimensional subspace $\Pi\subset T_mM$
containing $w$,  say $\Pi={\rm span}[w,u]$, then the following quantity
\begin{equation}\nonumber
K_F(w,\Pi) = \frac{g_F(w)\bigl({\bf R}_w(u),u\bigr)}{g_F(w)(w,w)g_F(w)(u,u)-g_F(w)(w,u)^2} 
\end{equation}
does not depend on $u$ but only on the flag $(w,\Pi)$. This is the so-called {\rm flag curvature} of the flag $(w,\Pi)$.
\end{propdef}
\begin{proof}
By Proposition \ref{propjacobicurvatureendo}, to be proved below, and Proposition \ref{propwronskianfundamentaltensor}, 
the statement about ${\bf R}_w$ is nothing but a manifestation of 
the symmetry of the Jacobi endomorphism stated in (3) of 
Proposition \ref{propfanninglagr}.
\end{proof}

\subsubsection{Proofs of Propositions \ref{propjacobicurvatureendo} and \ref{propdynamicalcovariantderivative}}\label{sectionproofs}

Let $V$ be a vector field on $M$. Since $V^\mathfrak{h}$ is $\pi$-related to $V$, then
\begin{equation}\label{eq341}
\mathcal{J}(V^\mathfrak{h}) = V^\mathfrak{v}. 
\end{equation}
Let, as in $\S$\ref{sectionliebrackets}, ${\rm H}:\mathcal{V}TM\rightarrow\mathcal{H}TM$ be the bundle isomorphism
corresponding to the horizontal derivative. From (\ref{inversehorizontalderivative}) and (\ref{fundamentalendoalmosttangent}) we have
${\rm H}^{-1}=-\mathcal{J}|_{\mathcal{H}TM}$. It follows from this and (\ref{eq341}) that
\begin{equation}\label{eq431}
{\rm H}(V^\mathfrak{v}) = -V^\mathfrak{h}. 
\end{equation}
Substituting this in (\ref{jacobi2}) gives us
\begin{equation}
\mathcal{K}(V^\mathfrak{v}) = \mathcal{P}_\mathcal{V}\bigl([S,V^\mathfrak{h}]\bigr). 
\end{equation}
Let us now compute ${\bf R}_{\dot{\gamma}(t)}(V)$. We have $\dot{\gamma}(t)^\mathfrak{v}=C_{\dot{\gamma}(t)}$ and $\dot{\gamma}(t)^\mathfrak{h}=S_{\dot{\gamma}(t)}$,
since ${\rm d}\pi(\dot{\gamma}(t))S=\dot{\gamma}(t)$ and $S$ is horizontal. Thus 
\begin{equation}\nonumber
{\bf R}_{\dot{\gamma}(t)}(V)={i_{\dot{\gamma}(t)}}^{-1}\mathcal{R}(S,V^\mathfrak{h})C.
\end{equation}
On the other hand, 
it follows from {\bf L.} that $\nabla_SC=\mathcal{P}_\mathcal{V}(S)=0$, $\nabla_{V^\mathfrak{h}}C=\mathcal{P}_\mathcal{V}(V^\mathfrak{h})=0$ and 
$\nabla_{[S,V^\mathfrak{h}]}C=\mathcal{P}_\mathcal{V}([S,V^\mathfrak{h}])$. Therefore,
\begin{equation}\nonumber
{\bf R}_{\dot{\gamma}(t)}(V)  = {i_{\dot{\gamma}(t)}}^{-1}\mathcal{P}_\mathcal{V}\bigl([S,V^\mathfrak{h}]\bigr) = 
{i_{\dot{\gamma}(t)}}^{-1}\mathcal{K}(V^\mathfrak{v}).
\end{equation} 
This proves Proposition \ref{propjacobicurvatureendo}.\par As for Proposition \ref{propdynamicalcovariantderivative}, note that since $\gamma$ is a geodesic of 
$S$ and $S$ is horizontal, then $\dot{\gamma}:I\subseteq\mathds{R}\rightarrow TM\backslash 0$ is a horizontal lift of $\gamma$ and, thus,
\begin{equation}\nonumber
{\rm D}^{\dot{\gamma}}V/{\rm d}t={i_{\dot{\gamma}}}^{-1}\nabla_S V^\mathfrak{v}. 
\end{equation}
On the other hand, by substituting $\mathcal{J}(S)=C$, $\mathcal{J}(V^\mathfrak{h}) = V^\mathfrak{v}$, and $\nabla_{V^\mathfrak{h}}C=0$
in the equality ${\rm T}(S,V^\mathfrak{h})=0$, we obtain $\nabla_SV^\mathfrak{v}=\mathcal{J}([S,V^\mathfrak{h}])$.
Therefore, since 
${\rm d}\Phi_{-t}(\Phi_t({\bf v}))V^\mathfrak{h}=-{\bf H}_v(t)a(t)$ (this follows from (\ref{eq431})), we have
\begin{eqnarray}
{\rm d}\Phi_{-t}(\Phi_t({\bf v}))\nabla_SV^\mathfrak{v} & = & -{\bf F}_v(t)\frac{\rm d}{{\rm d}t}\bigl(-{\bf H}_v(t)a(t)\bigr)\nonumber\\
& = & \frac{\rm d}{{\rm d} t}\bigl({\bf F}_v(t){\bf H}_v(t)a(t)\bigr)-\dot{\bf F}_v(t){\bf H}_v(t)a(t)\nonumber\\
& = & \dot{a}(t)-{\bf H}_v(t)a(t)\nonumber\\
& = & {\bf P}_{\ell_v}(t)\dot{a}(t),\nonumber
\end{eqnarray}
where we have used (\ref{inversehorizontalderivative}). The result follows. \qed

\section{An O'Neill formula for the flag curvatures in an isometric submersion via symplectic reduction of fanning curves}

In this section we shall see how 
a theory of symplectic reductions of fanning curves, as developed in \cite{Henrique}, 
leads to an O'Neill type formula for flag curvatures in a Finsler submersion. As remarked in the introduction, a similar theory of symplectic reductions has been
developed in \cite{ACZ} and applied to some problems from mechanics.
\subsection{Symplectic reduction of fanning curves}\label{subsectionreduction}

We begin by  
summarizing the results from \cite{Henrique} we shall need, and refer the reader to that work for more details.
\subsubsection{Linear symplectic reduction}

A subspace $\mathbb{W}\subseteq V$ is said to be {\it coisotropic} if $\mathbb{W}^\omega\subseteq \mathbb{W}$.
For such a subspace $\mathbb{W}$, 
the (restriction of) the symplectic form $\omega$ descends
to a symplectic form $\omega_R$ on $\mathbb{W}/\mathbb{W}^\omega$ and the symplectic space $(\mathbb{W}/\mathbb{W}^\omega,\omega_R)$ is the so-called linear symplectic 
reduction of $V$ by $\mathbb{W}$. Furthermore,
if $\ell\subset V$ is a Lagrangian subspace, then $\pi(\ell\cap \mathbb{W})$ is a Lagrangian subspace of $\mathbb{W}/\mathbb{W}^\omega$, where 
$\pi:\mathbb{W}\rightarrow \mathbb{W}/\mathbb{W}^\omega$ is the quotient map. 
We shall use the notation $\ell_R=\pi(\ell\cap \mathbb{W})$. Therefore, fixed a coisotropic subspace $\mathbb{W}$, one has a {\it symplectic reduction map} 
\begin{equation}\label{mapreduction}
\lambda:\Lambda(V)\rightarrow \Lambda(\mathbb{W}/\mathbb{W}^\omega)~,~~\lambda(\ell)=\ell_R.
\end{equation}
\noindent Consider the following open and dense subset $\mathcal{U}\subset\Lambda(V)$, $\mathcal{U}=\big\{\ell~:~\ell\cap \mathbb{W}^\omega=\{0\}\big\}$. 
For $\ell\in\mathcal{U}$, one has an isomorphism
\begin{equation}\label{isomorphismreduction}
\pi|_{\ell\cap \mathbb{W}}:\ell\cap \mathbb{W}\rightarrow \ell_R. 
\end{equation}
\begin{lemma}\label{lemmasmooth}
The map $(\ref{mapreduction})$ is smooth on $\mathcal{U}$. Furthermore, given $\ell\in\mathcal{U}$, upon identifying $\ell_R$ with 
$\ell\cap \mathbb{W}$ via $(\ref{isomorphismreduction})$, 
the derivative ${\rm d}\lambda(\ell):{\rm Bil}_{\rm sym}(\ell)\rightarrow{\rm Bil}_{\rm sym}(\ell\cap \mathbb{W})$ is the restriction map.
\end{lemma}

\subsubsection{The symplectic reduction of a fanning curve}\label{sectionsymplecticreductionfanning}

Let  $\mathbb{W}\subset V$ be a fixed  coisotropic subspace and  $\ell:I\subseteq\mathds{R}\rightarrow\Lambda(V)$ a fanning curve such that for all $t$,
\begin{itemize}
 \item[\bf i.] $\ell(t)\cap\mathbb{W}^\omega=\{0\}$, 
 \item[\bf ii.] the Wronskian $W(t)$ is non-degenerate on $\ell(t)\cap\mathbb{W}$.
\end{itemize}

\noindent In this setting, it follows from Lemma \ref{lemmasmooth} that the symplectic reduction of $\ell(t)$ by $\mathbb{W}$ is a smooth fanning curve
\begin{equation}\nonumber
\ell_R:=\lambda\circ\ell:I\rightarrow\Lambda(\mathbb{W}/\mathbb{W}^\omega).
\end{equation}

\begin{definition}
For each $t$, we let $\mathfrak{h}(t)$ be $\ell(t)\cap\mathbb{W}$, and let $\mathfrak{v}(t)\subset\ell(t)$ be its $W(t)$-orthogonal
subspace. Since $W(t)$ is non-degenerate on $\mathfrak{h}(t)$, then
\begin{equation}\label{decompositionsymplecticreduction}
\ell(t)=\mathfrak{h}(t)\oplus\mathfrak{v}(t). 
\end{equation}
With respect to the decomposition $V=\mathfrak{h}(t)\oplus\mathfrak{v}(t)\oplus h(t)$, the projectors onto $\mathfrak{h}(t)$ and $\mathfrak{v}(t)$ 
are denoted by ${\bf P}_\mathfrak{h}(t)$ and ${\bf P}_\mathfrak{v}(t)$, respectively.
\end{definition}

It follows from Lemma \ref{lemmasmooth} that for each $t$ the quotient map $\pi$ restricts to an isomorphism
\begin{equation}\label{iso5}
\pi|_{\mathfrak{h}(t)}:\mathfrak{h}(t)\rightarrow\ell_R(t) 
\end{equation}
that pulls back the Wronskian $W_R(t)$ of $\ell_R(t)$ to the restriction of $W(t)$ to $\mathfrak{h}(t)$.

\subsubsection{The O'Neill endomorphism}\label{sectiononeillendomorphism}

The set of fanning curves on $\Lambda(V)$ satisfying {\bf i.} and {\bf ii.} above is acted upon by the group 
${\rm SP}_\mathbb{W}(V)=\big\{{\bf T}\in{\rm SP}(V)\hskip 1pt:\hskip 1pt{\bf T}(\mathbb{W})=\mathbb{W}\big\}$ 
and so is the space ${\rm J}^1_{f,\mathbb{W}}(\mathds{R};\Lambda(V))$ 
of 1-jets of such curves.
A natural equivariant map 
\begin{equation}\nonumber
{\rm J}^1_{f,\mathbb{W}}(\mathds{R};\Lambda(V))\rightarrow\mathfrak{sp}_\mathbb{W}(V)
\end{equation}
is obtained by considering, for a given fanning curve $\ell(t)\in\Lambda(V)$ satisfying {\bf i.} and {\bf ii.},
the endomorphisms 
\begin{equation}\nonumber
{\bf F}_\mathfrak{h}(t):={\bf P}_\mathfrak{h}(t)\circ{\bf F}(t). 
\end{equation}
As for the first derivative $\dot{\bf F}_\mathfrak{h}(t)$, one has

\begin{lemma}\label{oneilllemma}
Let $\mathcal{A}(t)$ be a frame for $\ell(t)$. With respect to the basis $(\mathcal{A}(t),\mathcal{H}(t))$, the matrix of $\dot{\bf F}_\mathfrak{h}(t)$  
has the block form
\begin{equation}\nonumber
\left(
\begin{array}{cc}
-{\rm C}_1(t) & {\rm C}_2(t) \\
{\rm O} & {\rm C}_1(t)
\end{array}
\right),
\end{equation}
where ${\rm C}_1(t)$ is the matrix of ${\bf P}_\mathfrak{h}(t)|_{\ell(t)}$ in the basis $\mathcal{A}(t)$. As for the block ${\rm C}_2(t)$,
\begin{enumerate}
\item Denoting still by $W(t)$ the matrix of the Wronskian of $\ell(t)$ in the basis $\mathcal{A}(t)$, then 
${\rm C}_2(t)W(t)^{-1}$ is symmetric.
\item If $\mathcal{A}(t)=(\mathcal{A}_\mathfrak{h}(t),\mathcal{A}_\mathfrak{v}(t))$, where $\mathcal{A}_\mathfrak{h}(t)$ and 
$\mathcal{A}_\mathfrak{v}(t)$ are frames for $\mathfrak{h}(t)$ and $\mathfrak{v}(t)$, respectively, then
\begin{equation}\nonumber
\mathcal{A}(t){\rm C}_2(t)=\bigl({\bf P}_\mathfrak{v}(t)\dot{\mathcal{A}}_\mathfrak{h}(t),-{\bf P}_\mathfrak{h}(t)\dot{\mathcal{A}}_\mathfrak{v}(t)\bigr). 
\end{equation}
\end{enumerate}
\end{lemma}

\begin{definition}
The {\it O'Neill endomorphism}, at time $\tau$, of the pair $(\ell(t),\mathbb{W})$ is the $W(\tau)$-symmetric
endomorphism $${\bf A}(\tau):\ell(\tau)\rightarrow\ell(\tau)$$ whose matrix with respect to
a frame $\mathcal{A}(\tau)$ for $\ell(\tau)$ is the matrix ${\rm C}_2(\tau)$ from Lemma \ref{oneilllemma}. 
Therefore, given frames $\mathcal{A}_\mathfrak{h}(t)$ and 
$\mathcal{A}_\mathfrak{v}(t)$ for $\mathfrak{h}(t)$ and $\mathfrak{v}(t)$, respectively, then
\begin{eqnarray}
 {\bf A}(t)\mathcal{A}_\mathfrak{h}(t) & = & {\bf P}_\mathfrak{v}(t)\dot{\mathcal{A}}_\mathfrak{h}(t)\label{oneill1}\\
 {\bf A}(t)\mathcal{A}_\mathfrak{v}(t) & = & -{\bf P}_\mathfrak{h}(t)\dot{\mathcal{A}}_\mathfrak{v}(t).\label{oneill2}
\end{eqnarray}
\end{definition}

The importance of ${\bf A}(t)$ is described in the way it relates the Jacobi endomorphism ${\bf K}_R(t)$ of $\ell_R(t)$ with the 
``$\mathfrak{h}$-component'' of the Jacobi endomorphism ${\bf K}(t)$ of $\ell(t)$:

\begin{theorem}\label{oneillformula}
Given $a\in\mathfrak{h}(t)$, let $\overline{a}$ denote its image under the isomorphism $(\ref{iso5})$. Then,
\begin{equation}\nonumber
W_R(t)\bigl({\bf K}_R(t)\overline{a},\overline{a}\bigr)=W(t)\bigl({\bf K}(t)a,a\bigr)+3W(t)\bigl({\bf A}(t)a,{\bf A}(t)a\bigr). 
\end{equation}
\end{theorem}

\subsection{Isometric submersions of Finsler manifolds} \label{ssectionisometricsubs}

In this section we shall briefly collect some definitions and results from \cite{AD2}.

\begin{definition}
Given Finsler manifolds $(M,F_1)$ and $(N,F_2)$, a submersion
\begin{equation}\label{submersion}
f: M\rightarrow N
\end{equation}
is said to be {\it isometric} if, for every $m\in M$, the derivative ${\rm d}f(m):T_mM\rightarrow T_{f(m)}N$ maps the closed unit ball of $(F_1)_m$
onto the closed unit ball of $(F_2)_m$. 
\end{definition}

\begin{remark}
This concept can be alternatively stated as follows: for all $m\in M$, the derivative 
${\rm d}f(m):T_mM\rightarrow T_{f(m)}N$ induces an isometry between $T_{f(m)}N$ and 
the quotient
$T_mM / \ker {\rm d}f(m)$, endowed with the quotient norm
\[
|[v]|_{quotient} = \min_{w\in \ker {\rm d}f(m)} F_1(v+w) \, .
\]
\end{remark}

For an isometric submersion one defines the {\it horizontal cone} at a given $m$ as the set
\begin{equation}\nonumber
\mathscr{H}_m = \{v\in T_mM\backslash 0~:~F_1(v) = F_2({\rm d}f(m)v)\}\, , 
\end{equation}
that is, the elements of the horizontal cone are the non-zero vectors realizing the quotient norm above. 

Denoting by $\mathscr{V}_m$ the kernel of ${\rm d}f(m)$, one has, for each $v\in\mathscr{H}_m$, a $g_{F_1}(v)$-orthogonal decomposition
\begin{equation}\nonumber
T_mM = T_v\mathscr{H}_m\oplus\mathscr{V}_m 
\end{equation}
and the derivative ${\rm d}f(m)$ restricts to an isometry
\begin{equation}\label{isometry2}
 {\rm d}f(m) : \bigl(T_v\mathscr{H}_m,g_{F_1}(v)\bigr) \rightarrow \bigl( T_{f(m)}N, g_{F_2}(u)\bigr)
\end{equation}
for $u={\rm d}f(m)v$.
\par An immersed curve $\gamma:I\subseteq\mathds{R}\rightarrow M$ is said to be {\it horizontal} if $\dot{\gamma}(t)\in\mathscr{H}_{\gamma(t)}$ for every $t\in I$.
If $\gamma$ is a geodesic, this condition holds once it holds for some $t_0\in I$.

\subsection{The point of view of symplectic reductions}

A submanifold $P$ of a symplectic manifold $(Q,\omega)$ is co-isotropic if, for every $p\in P$, $T_pP$ is a co-isotropic subspace of $T_pQ$.
In this case, the distribution $p\mapsto T_pP^\omega$ on $P$ is integrable. When the space of leaves $P_R$ of the corresponding foliation has a smooth structure,
the pull-back of $\omega$ to $P$ descends to a symplectic structure on $P_R$; we refer to \cite{marsden} for more details.
This procedure has been applied in \cite{AD2} to obtain a symplectic description of an isometric submersion that, by passing from co-tangent to tangent bundles via the Legendre
transformations, goes as follows:

\begin{definition}
The {\it co-normal bundle} of the isometric submersion (\ref{submersion}) is the submanifold of $TM\backslash 0$ given by the union of all horizontal cones, 
and shall be denoted by $\mathcal{N}$. The derivative of $f$ restricts to a map
\begin{equation}\nonumber
\nu = f_*|_\mathcal{N} : \mathcal{N} \rightarrow TN\backslash 0. 
\end{equation}
\end{definition}

\begin{proposition}\label{propconormalreduction}
The co-normal bundle $\mathcal{N}$ is a co-isotropic submanifold of $(TM\backslash 0,\newline\omega_{F_1})$ with smooth space of leaves $\mathcal{N}_R$. The map $\nu$ above
is constant on the leaves and the induced map 
$\overline{\nu}:\mathcal{N}_R\rightarrow (TN\backslash 0,\omega_{F_2})$ is a symplectic diffeomorphism. Furthermore, the geodesic flow $\Phi_t^{F_1}$ of $F_1$ leaves $\mathcal{N}$
invariant and its restriction to $\mathcal{N}$ descends to a flow in $\mathcal{N}_R$ which corresponds, under $\overline{\nu}$, to the geodesic flow of $F_2$, $\Phi_t^{F_2}$.
\end{proposition}

In particular, it follows from the  proposition above that given $v\in\mathcal{N}$, and letting $u=f_*v$, the map 
\begin{equation}\nonumber
\lambda_v : \Lambda\bigl(T_v TM\bigr) \rightarrow \Lambda\bigl(T_uTN\bigr) , ~~ \lambda(\ell)={\rm d}\nu(v)(\ell\cap T_v\mathcal{N})
\end{equation}
is well-defined and is the symplectic reduction map (\ref{mapreduction}) with respect to the co-isotropic subspace $T_v\mathcal{N}\subset T_vTM$. Observe
that 
\begin{equation}\nonumber
\lambda_v(\mathcal{V}_vTM)=\mathcal{V}_uTN; 
\end{equation} 
indeed, this follows from the following lemma whose straightforward proof will be omitted.
\begin{lemma}\label{lemmasubmersions}
Let $m \in M$, $v\in\mathscr{H}_m$, and $u={ \rm d}f(m)v$. Then,
\begin{enumerate}
\item $\mathcal{V}_vTM \cap T_v\mathcal{N} = T_v\mathscr{H}_m$.
\item The map ${\rm d}\nu(v)|_{T_v\mathscr{H}_m}$ is equal to $i_u \circ {\rm d}f(m)|_{T_v\mathscr{H}_m} :T_v\mathscr{H}_m\rightarrow\mathcal{V}_uTN$.
\end{enumerate}
\end{lemma}

\subsection{The Jacobi curves}

We now compare the Jacobi curves of the total space and the base space of an isometric submersion, based on \cite{Henrique}. This will furnish the desired 
O'Neill formula. 

Let  $\gamma:I\subseteq\mathds{R}\rightarrow M$ be fixed a unit-speed horizontal geodesic, with $\dot{\gamma}(0)=v$, and consider, as in $\S$\ref{subsectionjacobisprays},
the Jacobi curves associated to $F_1$ and $F_2$, based at $v$ and $u=f_*v$, respectively,
\begin{equation}\nonumber
\ell_v(t)\in\Lambda(T_vTM) ~~,~~ \ell_u(t)\in\Lambda(T_uTN). 
\end{equation}

\begin{proposition}
We have that $\ell_u=\lambda_v\circ\ell_v$. 
\end{proposition}
\begin{proof}
This follows from the statement about the flows in Proposition \ref{propconormalreduction} and the fact
that $\lambda_w(\mathcal{V}_wTM)=\mathcal{V}_{f_*w}TN$ for all $w\in\mathcal{N}$.
\end{proof}

\begin{lemma}
For all $t$, $\ell_v(t)\cap T_v\mathcal{N}^\omega=\{0\}$. 
\end{lemma}
\begin{proof}
Since $T\mathcal{N}$ is invariant by the derivative of $\Phi_t^{F_1}$ the same is true of $T\mathcal{N}^\omega$. Therefore,
${\rm d}\Phi_t^{F_1}(v)\bigl(\ell_v(t)\cap T_v\mathcal{N}^\omega\bigr)=\mathcal{V}_{\dot{\gamma}(t)}TM\cap T_{\dot{\gamma}(t)}\mathcal{N}^\omega$.
Let us show that $\mathcal{V}_wTM\cap T_w\mathcal{N}^\omega=\{0\}$ for all $w\in\mathcal{N}$. On one hand, from the first part of Proposition \ref{propconormalreduction},
we have $T_w\mathcal{N}^\omega={\rm ker}\hskip 1pt{\rm d}\nu(w)$. On the other hand, 
$\mathcal{V}_wTM\cap T_w\mathcal{N}^\omega\subset\mathcal{V}_wTM\cap T_w\mathcal{N}=T_w\mathscr{H}_m$ and ${\rm ker}\hskip 1pt {\rm d}\nu(w)|_{T_w\mathscr{H}_m}=\{0\}$ since,
by Lemma \ref{lemmasubmersions},
${\rm d}\nu(w)|_{T_w\mathscr{H}_m}=i_{f_*w}\circ {\rm d}f(m)|_{T_w\mathscr{H}_m}$. The result follows.
\end{proof}

The above lemma says that the pair $(\ell_v(t),T_v\mathcal{N})$ fulfils the conditions in $\S$\ref{sectionsymplecticreductionfanning} (condition {\bf ii.}
automatically holds since the Wronskian $W_v(t)$ is positive-definite). Therefore, $\ell_v(t)$ decomposes as
\begin{equation}\label{decompositionjacobisubmersion}
\ell_v(t) = \mathfrak{h}_v(t)\oplus\mathfrak{v}_v(t)
\end{equation}
and the O'Neill endomorphism ${\bf A}_v(t)$ is defined.

\begin{lemma}
Under the isomorphism $(\ref{isomorphism3})$, the decomposition $(\ref{decompositionjacobisubmersion})$ corresponds to the decomposition
\begin{equation}\label{decomposition3}
T_{\gamma(t)}M=T_{\dot{\gamma}(t)}\mathscr{H}_{\gamma(t)}\oplus\mathscr{V}_{\gamma(t)}. 
\end{equation}
\end{lemma}
\begin{proof}
From ${\rm d}\Phi_t^{F_1}(v)T_v\mathcal{N}=T_{\dot{\gamma}(t)}\mathcal{N}$ and Lemma \ref{lemmasubmersions}, we obtain
\begin{equation}\nonumber
{\rm d}\Phi_t^{F_1}(v)\mathfrak{h}_v(t) = {\rm d}\Phi_t^{F_1}(v)\bigl(\ell_v(t)\cap T_v\mathcal{N}\bigr) =
\mathcal{V}_{\dot{\gamma}(t)}TM\cap T_{\dot{\gamma}(t)}\mathcal{N}=T_{\dot{\gamma}(t)}\mathscr{H}_{\gamma(t)}.
\end{equation}
This proves the assertion about $\mathfrak{h}_v(t)$. The assertion about $\mathfrak{v}_v(t)$ then follows since the decompositions (\ref{decomposition3}) and 
(\ref{decompositionjacobisubmersion})
are orthogonal with respect to $g_{F_1}(\dot{\gamma}(t))$ and $W_v(t)$, respectively, and these inner products correspond under (\ref{isomorphism3}). 
\end{proof}

Given a unit vector $w\in T_v\mathscr{H}_m$, with $g_{F_1}(v)(v,w)=0$, let us denote $a=i_v(w)\in\mathfrak{h}_v(0)$ and $\overline{a}={\rm d}\nu(v)a\in\ell_u(0)$.
From (2) of Lemma \ref{lemmasubmersions} we have $\overline{a}=i_u(f_*w)$ and, since (\ref{isometry2}) is an isometry,
$F_2(f_*w)=1$ and $g_{F_2}(u)(u,f_*w)=0$. On the one hand, denoting $\Pi={\rm span}[v,w]$ then Theorem \ref{theoremjacobiflag} gives us
\begin{eqnarray}
K_{F_1}(v,\Pi) & = & W_v(0)\bigl({\bf K}_v(0)a,a\bigr)\nonumber\\
K_{F_2}(u,f_*\Pi) & = & W_u(0)\bigl({\bf K}_u(0)\overline{a},\overline{a}\bigr).\nonumber
\end{eqnarray}
On the other hand, it follows from Theorem \ref{oneillformula} that
\begin{equation}\nonumber
W_u(0)\bigl({\bf K}_u(0)\overline{a},\overline{a}\bigr) = W_v(0)\bigl({\bf K}_v(0)a,a\bigr) + 3W_v(0)\bigl({\bf A}_v(0)a,{\bf A}_v(0)a\bigr). 
\end{equation}

\noindent Therefore, 

\begin{theorem}
Let $A(t)$ correspond to ${\bf A}_v(t)$ under $(\ref{isomorphism3})$. Then,
\begin{equation}\nonumber
 K_{F_2}(u,f_*\Pi)  =  K_{F_1}(v,\Pi) + 3g_{F_1}(v)\bigl(A(0)w,A(0)w\bigr). 
\end{equation}
\end{theorem}
Observe that the expressions (\ref{oneill1}) and (\ref{oneill2}) and Proposition \ref{propdynamicalcovariantderivative} imply that
\begin{equation}\nonumber
A(t)V(t)  =  P_\mathscr{V}(t)\Bigl(\frac{{\rm D}^{\dot{\gamma}}}{{\rm d}t}P_\mathscr{H}(t)V(t)\Bigr) 
- P_\mathscr{H}(t)\Bigl( \frac{{\rm D}^{\dot{\gamma}}}{{\rm d}t}P_\mathscr{V}(t)V(t) \Bigr), 
\end{equation}
where $V\in\mathfrak{X}(\gamma)$, and  $P_\mathscr{H}(t)$ and $P_\mathscr{V}(t)$ are the projections onto $T_{\dot{\gamma}(t)}\mathscr{H}_{\gamma(t)}$ 
and $\mathscr{V}_{\gamma(t)}$, respectively,
with respect to $(\ref{decomposition3})$.

\section{A dynamical characterization of the sign of flag curvature}

\begin{definition}
A Legendrian distribution $\mathcal{L}$ on $\Sigma_FM$ is said to have the {\it positive} (resp. {\it negative}) {\it twist property} if, for every 
$v\in\Sigma_FM$, the curve of Lagrangian subspaces
\begin{equation}\label{curve1}
t \mapsto {\rm d}\Phi_{-t}^F(\dot{\gamma}(t))\bigl(\mathcal{L}_{\dot{\gamma}(t)}\bigr) \in \Lambda\bigl({\rm ker}(\alpha_F)_v\bigr), 
\end{equation}
where $\gamma(t)$ is the geodesic with $\dot{\gamma}(0)=v$, has positive-definite (resp. negative-definite) Wronskian for all $t$.
\end{definition}

\begin{remark}
Pointing toward the Maslov index theory, the above property has the following reformulation:
over $\Sigma_FM$ there is a fiber bundle $\Lambda(\Sigma_FM)\rightarrow\Sigma_FM$ whose fiber over a given $v$ is $\Lambda({\rm ker}(\alpha_F)_v)$. Oberve that
the flow $\Phi_t^F$ lifts in a canonical way to a flow 
$\widehat{\Phi}^F_t : \Lambda(\Sigma_FM) \rightarrow \Lambda(\Sigma_FM)$.
Given a Legendrian distribution $\mathcal{L}$ on $\Sigma_FM$, its {\it Maslov cycle} is the subset $\Lambda_{\geq 1}(\mathcal{L})\subset\Lambda(\Sigma_FM)$,
\begin{equation}\nonumber
\Lambda_{\geq 1}(\mathcal{L}) = \big\{ (v,\ell)\hskip 1pt:\hskip 1pt \ell\cap \mathcal{L}_v\neq\{0\} \big\}. 
\end{equation}
This is a stratified submanifold of co-dimension 1 with a natural co-orientation given by using the identification (\ref{identificationwronskian}). 
The positive twist property for $\mathcal{L}$ is then equivalent to requiring that, for all $\ell\in\Lambda(\Sigma_FM)$, if the flow line of 
$\widehat{\Phi}^F_t$ through $\ell$ crosses $\Lambda_{\geq 1}(\mathcal{L})$, it does so 
pointing toward the co-orientation of $\Lambda_{\geq 1}(\mathcal{L})$. 
\end{remark}

We shall prove

\begin{proposition}
$(M,F)$ has positive (resp. negative) flag curvature if, and only if, the horizontal bundle $\mathcal{H}\Sigma_FM$ (see (\ref{horizontalcontact})) has 
the positive (resp. negative) twist property.
\end{proposition}

Positiveness (resp. negativeness) of the flag curvature means positiveness (resp. negativeness) of the quadratic form
\begin{equation}\nonumber
g_F(\dot{\gamma}(t))\bigl({\bf R}_{\dot{\gamma}(t)}\cdot\hskip 1pt,\hskip 1pt \cdot\bigr) : 
{\rm ker}\hskip 1pt g_F(\dot{\gamma}(t))\bigl(\dot{\gamma}(t)\hskip 1pt,\hskip 1pt\cdot\bigr) \rightarrow\mathds{R}
\end{equation}
for all $t$ and $v$. Recall $\S$\ref{sectionjacobicurvesfinsler}: the curve (\ref{curve1}) is 
the horizontal curve $h_v^{\rm c}(t)$ of $\mathscr{P}^{\rm c}$ when $\mathcal{L}=\mathcal{H}\Sigma_FM$, and the above quadratic form
corresponds to $W_v^{\rm c}(t)\bigl({\bf K}_v^{\rm c}(t)\cdot\hskip 1pt , \hskip 1pt \cdot\bigr)$ under (\ref{isomorphism5}).
Therefore, the above proposition follows at once of the following general property of fanning curves.

\begin{proposition}
Let $W_h(t)$ denote the Wronskian of the horizontal curve $h(t)$ of a fanning curve $\ell(t)\in\Lambda(V)$. 
Then, given $t$ and $u,v\in\ell(t)$,
\begin{equation}\nonumber
W_h(t)\bigl({\bf H}(t)u,{\bf H}(t)v\bigr)=W(t)\bigl({\bf K}(t)u,v\bigr). 
\end{equation}
\end{proposition}

\begin{proof}
By choosing linear symplectic coordinates, we can suppose $(V,\omega)=(\mathds{R}^{2n},\omega_0)$ where 
$\omega_0({\bf u},{\bf v})={\bf u}^T{\bf J}{\bf v}$ and {\bf J} is the standard complex structure of $\mathds{R}^{2n}$. 
Given a frame $\mathcal{A}(t)$ for $\ell(t)$, the matrices of $W(t)$ and $W_h(t)$ in the basis $\mathcal{A}(t)$ and $\mathcal{H}(t)$ are, respectively,
\begin{equation}\nonumber
\mathcal{A}(t)^T{\bf J}\dot{\mathcal{A}}(t)=-\dot{\mathcal{A}}(t)^T{\bf J}\mathcal{A}(t) ~~,~~\mathcal{H}(t)^T{\bf J}\dot{\mathcal{H}}(t);
\end{equation}
where above  we use that $\mathcal{A}(t)^T{\bf J}\mathcal{A}(t)=0$.
If the frame $\mathcal{A}(t)$ is normal, then $\mathcal{H}(t)=\dot{\mathcal{A}}(t)$ and 
$\ddot{\mathcal{A}}(t)=-(1/2)\mathcal{A}(t)\{\mathcal{A}(t),t\}$ and, therefore, 
\begin{equation}\nonumber
\mathcal{H}(t)^T{\bf J}\dot{\mathcal{H}}(t)=-(1/2)\dot{A}(t)^T{\bf J}\mathcal{A}(t)\{\mathcal{A}(t),t\}.
\end{equation}
On the other hand, since the matrix of ${\bf K}(t)|_{\ell(t)}$ in the basis $\mathcal{A}(t)$ is $(1/2)\{\mathcal{A}(t),t\}$ (cf. Proposition \ref{propjacobicurvatureendo}), 
the matrix of $W(t)\bigl({\bf K}(t)|_{\ell(t)}\cdot,\cdot\bigr)$ in the basis $\mathcal{A}(t)$ is given by
$-(1/2)\dot{\mathcal{A}}(t)^T{\bf J}A(t)\{\mathcal{A}(t),t\}$. This shows that the matrix of $W(t)\bigl({\bf K}(t)|_{\ell(t)}\cdot,\cdot\bigr)$ in a basis
$\mathcal{A}(t)$ is equal to the matrix of $W_h(t)$ in the basis $\mathcal{H}(t)={\bf H}(t)\mathcal{A}(t)$ provided that the frame $\mathcal{A}(t)$ is normal.
The result now follows from the fact that given $\tau$ and a basis $\mathcal{B}$ of $\ell(\tau)$, there is a unique normal frame
$\mathcal{A}(t)$ with $\mathcal{A}(\tau)=\mathcal{B}$.
\end{proof}

\section{The flag curvature of a class of projectively related Finsler metrics}

\subsection{Statement of the result}

Let $(M,F_0)$ be a Finsler manifold and $\theta$ a smooth 1-form on $M$ such that 
\begin{itemize}
\item[$(i)$] ${F_0}^*(\theta_m)<1$ for all $m\in M$,
\item[$(ii)$] ${\rm d}\theta=0$.
\end{itemize}
The first condition ensures that the following deformation of $F_0$,
\begin{equation}\nonumber
F=F_0+\theta:TM\rightarrow\mathds{R} ,
\end{equation}
defines a Finsler metric on $M$ (this follows, for instance, from the proof of Lemma \ref{lemma1} below), and the closedness of $\theta$ implies that $F$ and $F_0$ 
share the same unparametrized geodesics
since  the associated arc-length functionals have the same extremals. 
\par We shall prove the following relation between the flag curvatures of $F_0$ and $F$.

\begin{theorem}\label{thmprojective}
The map 
\begin{equation}\label{mappsi}
\Psi(v)=\mathscr{L}_{F_0}^{-1}\bigl(\mathscr{L}_F(v)-\theta\bigr)
\end{equation}
restricts to a diffeomorphism from $\Sigma_FM$ onto $\Sigma_{F_0}M$.
Given a $2$-plane $\Pi\subset T_mM$, $\Pi={\rm span}[{\bf v},{\bf w}]$, where ${\bf v}\in\Sigma_FM$ and 
${\bf w}\in T_{\bf v}(\Sigma_FM\cap T_mM)={\rm ker}\hskip 1pt g_F({\bf v})({\bf v},\cdot)$, let $\widetilde{\Pi}={\rm span}[{\bf u},\widetilde{\bf w}]\subset T_mM$ 
be the $2$-plane where 
${\bf u}=\Psi({\bf v})$ and $\widetilde{\bf w}={\rm d}\Psi({\bf v}){\bf w}$. Denoting by $\phi:\Sigma_{F_0}M\rightarrow\mathds{R}$ the function 
$\phi(z)=1/(1+\theta(z))$, then
\begin{equation}\label{curvaturesprojectivechange}
K_F({\bf v},\Pi) = \phi({\bf u})^2K_{F_0}({\bf u},\widetilde{\Pi})-\frac{1}{2}\Bigl[\frac{1}{2}S_{F_0}(\phi)^2-
\phi S_{F_0}(S_{F_0}(\phi))\Bigr]({\bf u}),
\end{equation}
where $S_{F_0}$ is the geodesic spray of $F_0$. Alternatively, if $h$ is a primitive for $\theta$ around $m$ and if we let $f(t)=t+h(\gamma_{\bf u}(t))$,
where $\gamma_{\bf u}$ is the $F_0$-geodesic
with $\dot{\gamma}_{\bf u}(0)={\bf u}$, then 
\begin{equation}\nonumber
K_F({\bf v},\Pi) = \frac{1}{\dot{f}(0)^2}\Bigl[K_{F_0}({\bf u},\widetilde{\Pi})-\frac{1}{2}\{f(t),t\}|_{t=0}\Bigr], 
\end{equation} 
where $\{f(t),t\} = (d/dt)(\dot{f}^{-1}\ddot{f})-(1/2)(\dot{f}^{-1}\ddot{f})^2 $ is the Schwarzian derivative of $f(t)$.
\end{theorem}

\begin{remark}
It follows easily from the definition of the map $\Psi$ that $\widetilde{\bf w}$ is determined by the equality
$g_{F_0}({\bf u})(\widetilde{\bf w},\cdot)=g_F({\bf v})({\bf w},\cdot)$.
\end{remark}

\subsection{Preliminaries}

Throughout, $S_{F_0}$, $S_F$, and $S_{F_0}^*$, $S_F^*$, shall denote the geodesic sprays and co-geodesic vector fields, respectively, of $F_0$ and
$F$, viewed as vector fields on $\Sigma_{F_0}M$, $\Sigma_FM$, and $\Sigma_{F_0}^*M$, $\Sigma_F^*M$.

\begin{lemma}\label{lemma1}
We have that 
\begin{equation}\label{levelmagnetic}
\Sigma_F^*M=\{\xi\in T^*M\hskip 1pt : \hskip 1pt F_0^*( \xi-\theta)=1\}=\theta+\Sigma_{F_0}^*M 
\end{equation}
\end{lemma}
\begin{proof}
Since $F_0^*(\theta)<1$, then $\theta+\Sigma_{F_0}^*M$ is the unit co-sphere bundle $\Sigma_{\hat{F}}^*M$ of some Finsler metric $\hat{F}$ on $M$. To see that
$\hat{F}=F$, let $v\in T_mM$ and compute:
\begin{eqnarray}
\hat{F}(v) & = & (\hat{F}^*)^*(v)\nonumber\\ &  = & {\rm sup}\Big\{\xi(v)\hskip 2pt : \hskip 2pt \xi\in\Sigma_{\hat{F}}^*M\cap T_mM=\theta_m+\Sigma_{F_0}^*M\cap T_m^*M\Big\}\nonumber\\
 & = &  \theta_m(v) + {\rm sup}\big\{\xi(v)\hskip 2pt : \hskip 2pt \xi\in \Sigma_{F_0}^*M\cap T_m^*M\big\}~=~\theta_m(v) + F_0(v).\nonumber
\end{eqnarray}
\end{proof}

If we introduce the {\it magnetic Hamiltonian} $H_\mathfrak{m}:T^*M\backslash 0\rightarrow\mathds{R}$,
\begin{equation}
H_\mathfrak{m}(\xi) = (1/2)(F_0^*)^2(\xi-\theta), 
\end{equation}
then (\ref{levelmagnetic}) says that the energy level $1/2$ of $H_\mathfrak{m}$ is
\begin{equation}
H_\mathfrak{m}^{-1}(1/2)=\Sigma_F^*M. 
\end{equation}
Let us follow the terminology in \cite[Chap. 3]{marsden}. The Hamiltonian 
$H_\mathfrak{m}$ corresponds to the Lagrangian function $L_\mathfrak{m}:TM\backslash 0\rightarrow\mathds{R}$,
$L_\mathfrak{m}(v)=(1/2)F_0(v)^2+\theta(v)$;
that is, the Legendre transformation $\mathscr{L}_\mathfrak{m}:TM\backslash 0\rightarrow T^*M\backslash 0$ of $L_\mathfrak{m}$, which one computes easily as
\begin{equation}\label{legendremagnetic}
\mathscr{L}_\mathfrak{m}(v)=\mathscr{L}_{F_0}(v)+\theta,
\end{equation}
is a diffeomorphism and $H_\mathfrak{m}\circ\mathscr{L}_\mathfrak{m}=E_\mathfrak{m}$, where the energy $E_\mathfrak{m}$ of
$L_\mathfrak{m}$ computes as $E_\mathfrak{m}=(1/2)(F_0)^2$. It follows that $\mathscr{L}_\mathfrak{m}$ restricts to a diffeomorphism
\begin{equation}\nonumber
\mathscr{L}_\mathfrak{m} : E_\mathfrak{m}^{-1}(1/2)=\Sigma_{F_0}M\longrightarrow H_\mathfrak{m}^{-1}(1/2)=\Sigma_F^*M 
\end{equation}
whose inverse, pre-composed with the diffeomorphism $\mathscr{L}_F:\Sigma_FM\rightarrow\Sigma^*_FM$, is the map (\ref{mappsi}):
\begin{equation}\label{mappsi2}
\Psi=\mathscr{L}_\mathfrak{m}^{-1}\circ\mathscr{L}_F:\Sigma_FM\rightarrow\Sigma_{F_0}M. 
\end{equation}

\begin{lemma}\label{lemmareparametrization}
If $\phi$ is the function in Theorem \ref{thmprojective}, then
\begin{equation}
\Psi_*S_F=\phi S_{F_0}.
\end{equation}
\end{lemma}
\begin{proof}
Let $X_{H_\mathfrak{m}}$ be the restriction to $H_\mathfrak{m}^{-1}(1/2)$ of the Hamiltonian vector field of $H_\mathfrak{m}$. Since the Hamiltonians
$H_\mathfrak{m}$ and $(1/2)(F^*)^2$ have the same energy level $1/2$, it follows easily that, on that level, their Hamiltonian vector fields must differ by
a multiplicative function $\lambda:\Sigma_F^*M\rightarrow\mathds{R}$, 
\begin{equation}\label{lambda}
S_F^*=\lambda X_{H_\mathfrak{m}}.
\end{equation}
On the other hand, $S_F^*$ is $\mathscr{L}_F$-related to $S_F$ and, letting $X_{E_\mathfrak{m}}$ be the restriction to $E_\mathfrak{m}^{-1}(1/2)=\Sigma_{F_0}M$ of
the Euler-Lagrange vector field of $L_\mathfrak{m}$, $X_{H_\mathfrak{m}}$ is $\mathscr{L}_\mathfrak{m}$-related to $X_{E_\mathfrak{m}}$. Therefore,
$\Psi_*S_F = (\lambda\circ\mathscr{L}_\mathfrak{m}) X_{E_\mathfrak{m}}$. It remains to show that $X_{E_\mathfrak{m}}=S_{F_0}$ and 
$\lambda\circ\mathscr{L}_\mathfrak{m}=\phi$. The former is a consequence of the closedness of $\theta$ since $L_\mathfrak{m}$ differs from $(1/2)(F_0)^2$ by 
$\theta$ (cf. \cite[Prop. 3.5.18]{marsden}). As for the latter, applying the canonical 1-form $\alpha$ to (\ref{lambda}), and recalling that $\alpha(S_F^*)=1$,
then 
\begin{equation}\label{lambda2}
\lambda=1/\alpha(X_{H_\mathfrak{m}}).
\end{equation}
On the other hand, since $X_{H_\mathfrak{m}}$ and $X_{E_\mathfrak{m}}$ are $\mathscr{L}_\mathfrak{m}|_{\Sigma_{F_0}M}$-related,
and
\begin{equation}\label{pullbackmagnetic}
{\mathscr{L}_\mathfrak{m}}^*\alpha={\mathscr{L}_{F_0}}^*\alpha+\pi^*\theta
\end{equation}
as follows from (\ref{legendremagnetic}), we have for $v\in\Sigma_{F_0}M$,
\begin{eqnarray}
\alpha_{\mathscr{L}_\mathfrak{m}(v)}(X_{H_\mathfrak{m}}) & = & ({\mathscr{L}_\mathfrak{m}}^*\alpha)_v(X_{E_\mathfrak{m}}) = 
({\mathscr{L}_{F_0}}^*\alpha)_v(X_{E_\mathfrak{m}})+ (\pi^*\theta)_v(X_{E_\mathfrak{m}})\nonumber\\
& = & g_{F_0}(v)\bigl(v\hskip 1pt, \hskip 1pt{\rm d}\pi(v)X_{E_\mathfrak{m}}\bigr) + \theta\bigl({\rm d}\pi(v)X_{E_\mathfrak{m}}\bigr).\nonumber
\end{eqnarray}
Using now that $X_{E_\mathfrak{m}}$ is a SODE, the above expression is $g_{F_0}(v)(v,v)+\theta(v)=1+\theta(v)$ and the equality
$\lambda\circ\mathscr{L}_\mathfrak{m}|_{\Sigma_{F_0}M}=\phi$ follows now from (\ref{lambda2}).
\end{proof}

\subsection{Proof of Theorem \ref{thmprojective}}

Consider, as in $\S$\ref{sectionjacobicurvesfinsler}, the Jacobi curves
\begin{equation}\nonumber
\ell_{\bf v}^{\rm c}(t)\in\Lambda\bigl({\rm ker}(\alpha_F)_{\bf v},\omega_F\bigr)~,~~ \ell_{\bf u}^{\rm c}(t)\in\Lambda\bigl({\rm ker}(\alpha_{F_0})_{\bf u},\omega_{F_0}\bigr), 
\end{equation}
based at {\bf v} and {\bf u}, associated to $F$ and $F_0$, respectively.  
We shall break up the proof in several simple steps.
\vskip 5pt

{\bf I.} The map $\Psi:(\Sigma_FM,\alpha_F)\rightarrow(\Sigma_{F_0}M,\Psi_*\alpha_F)$ is a fiber-preserving exact contact diffeomorphism and, 
by Lemma \ref{lemmareparametrization}, $\Psi_*S_F=\phi S_{F_0}$. Moreover, ${\rm d}(\Psi_*\alpha_F)=\omega_{F_0}$; for, 
it follows successively from the definition (\ref{mappsi2}) of $\Psi$, the definitions of $\alpha_F$, $\alpha_{F_0}$, and (\ref{pullbackmagnetic}) that
$\Psi_*\alpha_F={\mathscr{L}_\mathfrak{m}}^*\bigl((\mathscr{L}_F)_*\alpha_F\bigr)={\mathscr{L}_\mathfrak{m}}^*\alpha=\alpha_{F_0}+\pi^*\theta$, hence
${\rm d}(\Psi_*\alpha_F)={\rm d}\hskip 1pt \alpha_{F_0}+\pi^*{\rm d}\hskip 1pt \theta={\rm d}\hskip 1pt \alpha_{F_0}$.
Therefore, the derivative ${\rm d}\Psi({\bf v})$ restricts to a symplectic isomorphism
\begin{equation}\label{derivativepsi}
{\rm d}\Psi({\bf v}):\bigl({\rm ker}(\alpha_F)_{\bf v},\omega_F\bigr)\rightarrow\bigl({\rm ker}(\Psi_*\alpha_F)_{\bf u},\omega_{F_0}\bigr) 
\end{equation}
that maps $\ell_{\bf v}^{\rm c}(t)$ to the Jacobi curve $\hat{\ell}_{\bf u}^{\rm c}(t)\in\Lambda\bigl({\rm ker}(\Psi_*\alpha_F)_{\bf u},\omega_{F_0}\bigr)$ of the 
moving plane $\bigl({\rm ker}\hskip 1pt \Psi_*\alpha_F,\mathcal{V}\Sigma_{F_0}M,\Phi_t^{\phi S_{F_0}}\bigr)$ defined on the exact contact manifold 
$(\Sigma_{F_0}M,\Psi_*\alpha_F)$.

\vskip 5pt

{\bf II.} The flow $\Phi_t^{\phi S_{F_0}}$ is a {\it reparametrization} of $\Phi_t^{S_{F_0}}$; more precisely, if $\eta_u(t)=\eta(t,u)$ denotes
the solution, defined for $(t,u)$ on some neighborhood of $\{0\}\times\Sigma_{F_0}M$, to
\begin{equation}\label{equationeta}
\frac{\partial \eta}{\partial t}(t,u)=\phi\bigl(\Phi_{\eta(t,u)}^{S_{F_0}}(u)\bigr)~,~~\eta(0,u)=0, 
\end{equation}
then 
\begin{equation}\label{eq2}
\Phi_t^{\phi S_{F_0}}(u)=\Phi_{\eta(t,u)}^{S_{F_0}}(u).
\end{equation}
It follows from this and from a straightforward computation that
the derivative of $\Phi_{-t}^{\phi S_{F_0}}$ at $\Phi_t^{\phi S_{F_0}}({\bf u})$ takes the form
\begin{equation}\label{derivativesflows}
{\rm d}\hskip 1pt \Phi_{-t}^{\phi S_{F_0}}(\Phi_t^{\phi S_{F_0}}({\bf u})) = {\rm d}\hskip 1pt \Phi_{-\eta(t,{\bf u})}^{S_{F_0}}(\Phi_{\eta(t,{\bf u})}^{S_{F_0}}({\bf u}))
+\zeta\otimes (S_{F_0})_{\bf u}
\end{equation}
for some $\zeta\in T_{\bf u}^*(\Sigma_{F_0}M)$.

\vskip 5pt

{\bf III.} Let ${\rm pr}_{\bf u}:T_{\bf u}\Sigma_{F_0}M\rightarrow{\rm ker}(\alpha_{F_0})_{\bf u}$ be the projection map with 
kernel generated by $(S_{F_0})_{\bf u}$.
Since one also has 
$T_{\bf u}\Sigma_{F_0}M={\rm ker}(\Psi_*\alpha_F)_{\bf u}\oplus{\rm span}[(S_F)_{\bf u}]$ and $S_{F_0}$ generates the kernel of $\omega_{F_0}$, then
${\rm pr}_{\bf u}$ restricts to a symplectic isomorphism
\begin{equation}\label{pru}
{\rm pr}_{\bf u}:\bigl({\rm ker}(\Psi_*\alpha_F)_{\bf u},\omega_{F_0}\bigr)\rightarrow\bigl({\rm ker}(\alpha_{F_0})_{\bf u},\omega_{F_0}\bigr). 
\end{equation}
Recalling the definitions of $\ell_{\bf u}^{\rm c}(t)$ and $\hat{\ell}_{\bf u}^{\rm c}(t)$, it follows from (\ref{derivativesflows}) that 
${\rm pr}_{\bf u}(\hat{\ell}_{\bf u}^{\rm c}(t))=\ell_{\bf u}^{\rm c}(\eta_{\bf u}(t))$. Therefore, the composition of (\ref{derivativepsi}) with
(\ref{pru}),  
\begin{equation}\nonumber
{\bf T}={\rm pr}_{\bf u}\circ {\rm d}\Psi({\bf v}):\bigl({\rm ker}(\alpha_F)_{\bf v},\omega_F\bigr)\rightarrow\bigl({\rm ker}(\alpha_{F_0})_{\bf u},\omega_{F_0}\bigr),
\end{equation}
is a symplectic isomorphism such that
\begin{equation}\label{relationjacobicurves}
{\bf T}\ell_{\bf v}^{\rm c}(t) = \ell_{\bf u}^{\rm c}(\eta_{\bf u}(t)). 
\end{equation}

{\bf IV.} Observe that ${\rm pr}_{\bf u}$ is the identity on $\mathcal{V}_{\bf u}\Sigma_{F_0}M$, so ${\bf T}{\bf w}=\widetilde{\bf w}$. Applying 
Proposition \ref{transformationproperties} to
(\ref{relationjacobicurves}) one obtains
\begin{equation}
\frac{W_{\bf v}^{\rm c}(0)\bigl({\bf K}_{\bf v}^{\rm c}(0){\bf w},{\bf w}\bigr)}{W_{\bf v}^{\rm c}(0)({\bf w},{\bf w})} 
 =  \dot{\eta}_{\bf u}(0)^2 \frac{W_{\bf u}^{\rm c}(0)\bigl({\bf K}_{\bf u}^{\rm c}(0)\widetilde{\bf w},\widetilde{\bf w}\bigr)}{W_{\bf u}^{\rm c}(\widetilde{\bf w},\widetilde{\bf w})}
+\frac{1}{2}\{\eta_{\bf u}(t),t\}|_{t=0}
\end{equation}
and therefore
\begin{equation}\nonumber
K_F({\bf v},\Pi) =  \dot{\eta}_{\bf u}(0)^2 K_{F_0}({\bf u},\widetilde{\Pi})+\frac{1}{2}\{\eta_{\bf u}(t),t\}|_{t=0}.
\end{equation}
It remains to compute $\dot{\eta}_{\bf u}(0)$ and $\{\eta_{\bf u}(t),t\}|_{t=0}$. From (\ref{equationeta}) and (\ref{eq2}) one has
$\dot{\eta}_{\bf u}(t)=\phi(\Phi_t^{\phi S_{F_0}}({\bf u}))$. Hence,
\begin{itemize}
\item[$(i)$] $\dot{\eta}_{\bf u}(0)=\phi(\bf u)$
\item[$(ii)$] $\ddot{\eta}_{\bf u}(0)=\phi S_{F_0}(\phi)|_{\bf u}$
\item[$(iii)$] $\dddot{\eta}_{\bf u}(0)=\phi S_{F_0}\bigl(\phi S_{F_0}(\phi)\bigr)|_{\bf u}=\phi S_{F_0}(\phi)^2|_{\bf u}+\phi^2 S_{F_0}\bigl(S_{F_0}(\phi)\bigr)|_{\bf u}$
\end{itemize}
Therefore $\{\eta_{\bf u}(t),t\}|_{t=0}=\phi S_{F_0}\bigl(S_{F_0}(\phi)\bigr)|_{\bf u}-(1/2)S_{F_0}(\phi)^2|_{\bf u}$ and (\ref{curvaturesprojectivechange}) follows.

\section{The flag curvature of Katok perturbations}

Let $(M,F)$ be a Finsler manifold and $V$ a vector field on $M$ such that $F(V_m)<1$ for all $m\in M$. Regarding $V$ as a function
\begin{equation}\label{hamiltonianV}
V : T^*M \rightarrow \mathds{R},~~~V(\xi)=\xi(V_{\tau(\xi)}), 
\end{equation}
there exists a unique Finsler metric $\widehat{F}$ on $M$ whose dual $\widehat{F}^*$ is given by
\begin{equation}\nonumber
\widehat{F}^* = F^* + V.
\end{equation}

\begin{definition}
In the case where $V$ is a Killing vector field for $F$, that is, its flow $\Phi_t^V$ satisfies $(\Phi_t^V)^*F=F$ for all $t$, we shall call
$\widehat{F}$ the {\it Katok perturbation} of $F$ by $V$. 
\end{definition}

Although the computations of the flag curvature in the more general cases of perturbations by {\it homothetic} vector fields and even for  
{\it conformal} vector fields have been done (\cite{henrique3} and \cite{HuangMo}, resp.), a proof via fanning curves of the theorem below
is particularly simple and elegant and shall, thus, be presented here.
\begin{theorem}[Foulon \cite{FR}]\label{theoremkatok}
Let $\widehat{F}$ be a Katok perturbation of $F$. If $K_F\equiv 1$, then $K_{\widehat{F}}\equiv 1$.  
\end{theorem}

\subsection{Proof of Theorem \ref{theoremkatok}}

We shall denote by $\alpha_F$ and $\alpha_{\widehat{F}}$ the contact 1-forms on $\Sigma_F^*M$ and $\Sigma_{\widehat{F}}^*M$, respectively, and 
let $\omega_F=-{\rm d}\alpha_F$ and $\omega_{\widehat{F}}=-{\rm d}\alpha_{\widehat{F}}$. 
Let
$X$ be the Hamiltonian vector field of (\ref{hamiltonianV}). As pointed out in \cite{Ziller}, the Hamiltonian flow $\Phi_t^X$ is pulling-back by $\Phi_t^V$,
\begin{equation}\label{flowX}
\Phi_t^X = \bigl( {\rm d}\Phi_t^V \bigr)^*:T^*M\rightarrow T^*M.
\end{equation}
Since $V$ is a Killing vector field of $F$, it follows that $\Sigma_F^*M$ is invariant by $\Phi_t^X$ and, hence, $X$ is tangent to $\Sigma_F^*M$.
Also, since $F^*$ is constant on the orbits of $X$, we have the commutation of the flows $\Phi_t^{S_F^*}$ and $\Phi_t^X$,
\begin{equation}\label{commutation}
[S_F^*,X] = 0. 
\end{equation}
We shall still denote by $X$ and $V$ the restrictions of $X$ and (\ref{hamiltonianV}) to $\Sigma_F^*M$.
\par Consider the diffeomorphism
\begin{equation}\label{diffeomorphism7}
 \Psi : \Sigma_{\widehat{F}}^*M \rightarrow \Sigma_F^*M,~~~\Psi(\xi)=\frac{1}{F^*(\xi)}\xi.
\end{equation}

From the definitions, one easily computes
\begin{equation}\label{eq444}
\Psi_*\alpha_{\widehat{F}} = \frac{1}{\widehat{F}^*}\alpha_F.  
\end{equation}

\begin{lemma}\label{lemma456}
We have that $\Psi_*S_{\widehat{F}}^* = S_F^*+X$. 
\end{lemma}
\begin{proof}
All we have to show is that 
\begin{equation}\label{eq576}
i_{S_F^*+X}{\rm d}(\Psi_*\alpha_{\widehat{F}}) = 0~~,~~(\Psi_*\alpha_F)(S_F^*+X) = 1.  
\end{equation}
Observe that $\alpha_F(X)=V$ since (\ref{flowX}) implies that $X$ is $\tau$-related to the vector field $V$. Thus, since $\alpha_F(S_F^*)=1$ and, as functions on 
$\Sigma_F^*M$, $\widehat{F}^*=F^*+V=1+V$, the second equality in (\ref{eq576})
follows from (\ref{eq444}). By taking derivatives in (\ref{eq444}), and using that $i_{S_F^*}\omega_F=0$, $i_X\omega_F={\rm d}V$, and $\alpha_F(S_F^*+X)=\widehat{F}^*$, we obtain
sucessively,
\begin{eqnarray}
 i_{S_F^*+X}{\rm d}(\Psi_*\alpha_{\widehat{F}}) & = & (1/\widehat{F}^*)^2i_{S_F^*+X}({\rm d}V\wedge\alpha_F) + (1/\widehat{F}^*)i_{S_F^*+X}\omega_F\nonumber\\
 & = & (1/\widehat{F}^*)^2\bigl((S_F^*+X)(V)\alpha_F - \alpha_F(S_F^*+X){\rm d}V\bigr) + (1/\widehat{F}^*)i_X\omega_F \nonumber\\
 & = & (1/\widehat{F}^*)^2\bigl(S_F^*(V)\alpha_F - \widehat{F}^*{\rm d}V \bigr)  + (1/\widehat{F}^*){\rm d}V.\nonumber
\end{eqnarray}
On the other hand, the commutativity of the flows $\Phi_t^{S_F^*}$ and $\Phi_t^X$ gives us $S_F^*(V)=0$. The result follows.
\end{proof}

The lemma above and (\ref{commutation}) imply, respectively, 
\begin{equation}\nonumber
\Psi\circ \Phi_t^{S_{\widehat{F}}^*} = \Phi_t^{S_F^*+X}\circ\Psi = \bigl(\Phi_t^X\circ\Phi^{S_F^*}_t\bigr)\circ \Psi .
\end{equation}
On the other hand, $\Psi$ is fiber-preserving, and the same is true of $\Phi_t^X$ since it is $\tau$-related to a flow on $M$. 
Therefore, if $\ell_\xi^{\rm c}(t)\in\Lambda({\rm ker}(\alpha_{\widehat{F}})_\xi)$ and 
$\ell_\eta^{\rm c}(t)\in\Lambda({\rm ker}(\alpha_F)_\eta)$ denote, as in $\S$\ref{sectionjacobicurvesfinsler}, the Jacobi curves associated to $\widehat{F}$ and $F$, respectively, based
at $\xi\in\Sigma_{\widehat{F}}^*M$ and $\eta=\Psi(\xi)$, we have shown

\begin{proposition}
${\rm d}\Psi(\xi)$ restricts to an isomorphism 
${\bf T}:{\rm ker}(\alpha_{\widehat{F}})_\xi\rightarrow{\rm ker}(\alpha_F)_\eta$ such that 
\begin{equation}\label{eq678}
{\bf T}\ell_\xi^{\rm c}(t)=\ell_\eta^{\rm c}(t).
\end{equation}
\end{proposition}

\begin{proof}[Proof of Theorem \ref{theoremkatok}]
The hypothesis $K_F\equiv 1$ means that ${\bf K}_\eta^{\rm c}(t)\equiv{\bf Id}$ for all $\eta\in\Sigma_F^*M$. Applying Proposition \ref{transformationproperties} to (\ref{eq678}),
we obtain ${\bf K}_\xi^{\rm c}(t)\equiv{\bf Id}$ for all $\xi\in\Sigma_{\widehat{F}}^*M$ and, therefore, $K_{\widehat{F}}\equiv 1$.
\end{proof}

\end{document}